\documentclass[a4paper, 10pt, reqno]{amsart}
\usepackage[utf8]{inputenc}
\usepackage[T1]{fontenc}
\usepackage[largesc]{newtxtext}
\usepackage{newtxmath}              
\usepackage{microtype}
\usepackage{mathtools}              
\mathtoolsset{showonlyrefs}
\usepackage{amsthm}
\usepackage{amsfonts}
\usepackage{thmtools}
\usepackage{relsize}
\usepackage[mathscr]{eucal}
\usepackage[linktocpage]{hyperref}
\usepackage[sort,capitalize]{cleveref}
\usepackage{enumitem}			
\usepackage[dvipsnames]{xcolor}
\usepackage[msc-links, abbrev, non-sorted-cites]{amsrefs}
\usepackage{todonotes}
\usepackage{caption}
\SetMathAlphabet{\mathsf}{normal}{OT1}{lmss}{m}{n}
\SetMathAlphabet{\mathsf}{bold}{OT1}{lmss}{bx}{n}

\makeatletter
\def\nonumberfootnote{\xdef\@thefnmark{}\@footnotetext}			
\makeatother
\definecolor{colorred}{HTML}{B00000}
\definecolor{colorgreen}{HTML}{258300}
\definecolor{colorblue}{HTML}{2e32fa}
\definecolor{coloryellow}{HTML}{cbbb1a}
\hypersetup{colorlinks=true, linkcolor=colorred, citecolor=colorgreen, urlcolor=coloryellow}	
\numberwithin{equation}{section}
\newcommand{\rma}{{\ensuremath{\mathrm{a}}}}
\newcommand{\rmb}{{\ensuremath{\mathrm{b}}}}

\newcommand{\rmd}{{\ensuremath{\mathrm{d}}}}
\newcommand{\rme}{{\ensuremath{\mathrm{e}}}}

\newcommand{\rmA}{{\ensuremath{\mathrm{A}}}}
\newcommand{\rmB}{{\ensuremath{\mathrm{B}}}}

\newcommand{\rmI}{{\ensuremath{\mathrm{I}}}}

\newcommand{\sfh}{{\ensuremath{\mathsf{h}}}}

\newcommand{\sfp}{{\ensuremath{\mathsf{p}}}}

\newcommand{\sfL}{{\ensuremath{\mathsf{L}}}}

\newcommand{\scrF}{{\ensuremath{\mathscr{F}}}}

\newcommand{\N}{\boldsymbol{\mathrm{N}}}						
\newcommand{\R}{\boldsymbol{\mathrm{R}}}						
\renewcommand{\S}{\boldsymbol{\mathrm{S}}}						
\renewcommand{\d}{\,\mathrm{d}}				

\let\div\undefined
\DeclareMathOperator{\div}{div}				

\theoremstyle{definition}
\newtheorem{bump}{Bump}[section]
\theoremstyle{plain}
\newtheorem{theorem}[bump]{Theorem}
\newtheorem{proposition}[bump]{Proposition}
\newtheorem{definition}[bump]{Definition}
\newtheorem{lemma}[bump]{Lemma}
\newtheorem{corollary}[bump]{Corollary}

\theoremstyle{remark}
\newtheorem{remark}[bump]{Remark}
\newtheorem{example}[bump]{Example}

\crefname{theorem}{Theorem}{Theorems}
\crefname{proposition}{Proposition}{Propositions}
\crefname{definition}{Definition}{Definitions}
\crefname{lemma}{Lemma}{Lemmas}
\crefname{corollary}{Corollary}{Corollaries}
\crefname{hypothesis}{Hypothesis}{Hypotheses}
\crefname{remark}{Remark}{Remarks}
\crefname{example}{Example}{Examples}
\crefname{notation}{Notation}{Notations}

\renewenvironment{example}
  {\begin{oldexample}}
  {\hfill $\blacksquare$\end{oldexample}}
\renewenvironment{remark}
  {\begin{oldremark}}
  {\hfill $\blacksquare$\end{oldremark}}
\crefformat{section}{{§}#2#1#3}
\crefformat{subsection}{{§}#2#1#3}
\crefformat{subsubsection}{{§}#2#1#3}
\crefformat{appendix}{{§}#2#1#3}
\crefmultiformat{theorem}{Theorems #2#1#3}{ and #2#1#3}{, #2#1#3}{, and #2#1#3}
\crefmultiformat{proposition}{Propositions #2#1#3}{ and #2#1#3}{, #2#1#3}{, and #2#1#3}
\crefmultiformat{definition}{Definitions #2#1#3}{ and #2#1#3}{, #2#1#3}{, and #2#1#3}
\crefmultiformat{lemma}{Lemmas #2#1#3}{ and #2#1#3}{, #2#1#3}{, and #2#1#3}
\crefmultiformat{corollary}{Corollaries #2#1#3}{ and #2#1#3}{, #2#1#3}{, and #2#1#3}
\crefmultiformat{hypothesis}{Hypotheses #2#1#3}{ and #2#1#3}{, #2#1#3}{, and #2#1#3}
\crefmultiformat{remark}{Remarks #2#1#3}{ and #2#1#3}{, #2#1#3}{, and #2#1#3}
\crefmultiformat{example}{Examples #2#1#3}{ and #2#1#3}{, #2#1#3}{, and #2#1#3}
\crefmultiformat{notation}{Notations #2#1#3}{ and #2#1#3}{, #2#1#3}{, and #2#1#3}
\crefmultiformat{section}{{§§}#2#1#3}{ and #2#1#3}{, #2#1#3}{, and #2#1#3}
\crefmultiformat{subsection}{{§§}#2#1#3}{ and #2#1#3}{, #2#1#3}{, and #2#1#3}
\crefmultiformat{subsubsection}{{§§}#2#1#3}{ and #2#1#3}{, #2#1#3}{, and #2#1#3}
\crefrangeformat{equation}{#3\textcolor{black}{(}#1\textcolor{black}{)}#4 to #5\textcolor{black}{(}#2\textcolor{black}{)}#6}
\newcommand{\mms}{\mathsf{M}}				

\newcommand{\Leb}{\mathscr{L}}				
\newcommand{\Prob}{\mathscr{P}}		        

\newcommand{\push}{\sharp}					

\newcommand{\One}{1}

\allowdisplaybreaks

\let\oldtocsection=\tocsection
\let\oldtocsubsection=\tocsubsection
\let\oldtocsubsubsection=\tocsubsubsection
\renewcommand{\tocsection}[2]{\hspace{0em}\oldtocsection{#1}{#2}}
\renewcommand{\tocsubsection}[2]{\hspace{1em}\oldtocsubsection{#1}{#2}}
\renewcommand{\tocsubsubsection}[2]{\hspace{2em}\oldtocsubsubsection{#1}{#2}}
\newcommand{\nocontentsline}[3]{}
\newcommand{\tocless}[2]{\bgroup\let\addcontentsline=\nocontentsline#1{#2}\egroup}
\newcommand{\mres}{\mathbin{\vrule height 1.6ex depth 0pt width 0.13ex\vrule height 0.13ex depth 0pt width 1.3ex}}
\DeclareMathOperator{\sgn}{sgn}

\newcommand{\Law}{\mathrm{Law}}
\allowdisplaybreaks
\makeatletter
\@namedef{subjclassname@2020}{\textup{2020} Mathematics Subject Classification}
\makeatother
\newcommand{\PPP}{{\boldsymbol{\mathrm{P}}}}
\newcommand{\EEE}{{\boldsymbol{\mathrm{E}}}}
\newcommand{\rO}{r^o}

\begin{document}
\title[The radial part of $p$-Brownian motion]{The radial part of $p$-Brownian motion}

\author{Mathias Braun}


\address{Institute of Mathematics, EPFL, 1015 Lausanne, Switzerland}


\email{\href{mailto:mathias.braun@epfl.ch}{mathias.braun@epfl.ch}}





\thanks{Financial support by the EPFL through a Bernoulli Instructorship is gratefully acknowledged. I sincerely thank Michael Röckner for his kind invitation to the University of Bielefeld and numerous inspiring discussions.}

\begin{abstract} We initiate a geometric theory of $\smash{p}$-Brownian motion, the nonlinear Markov process associated with the $p$-Laplacian introduced by Barbu--Rehmeier--Röckner. More precisely, we analyze its radial processes thoroughly, relative to an arbitrary center and in every dimension. On the one hand, we show explicit Tanaka--Meyer semimartingale formulas; our consequential characterization of nontriviality of the associated local times reveals notable differences to classical Brownian motion. In parallel, we establish sharp exit time estimates as well as scaling estimates for the self-similarly rescaled radial process. We also prove isometry of the corresponding marginal laws in each Wasserstein distance. Our contributions equally cover the Leibenson process --- the nonlinear Markov process associated with the ``porous medium equation'' with $p$-Laplacian --- recently introduced by Barbu--Grube--Rehmeier--Röckner.
\end{abstract}

\maketitle

\thispagestyle{empty}

\tableofcontents

\addtocontents{toc}{\protect\setcounter{tocdepth}{2}}

\section{Introduction}

\subsection{Aims} The goal of this work is to initiate a geometric theory of \emph{$p$-Brownian motion}. We envision a unification of two research streams: the young theory of nonlinear Markov processes, recalled in \cref{Sub:pBMIntro}, and the celebrated interplay of classical Brownian motion with Riemannian geometry, recalled in \cref{Sub:ClassBM}. The process in question, recently introduced by Barbu--Rehmeier--Röckner \cite{barbu-rehmeier-rockner2026}, is associated to the parabolic $p$-Laplace equation
\begin{align}\label{Eq:pLaplaceequation}
\frac{\partial}{\partial t}u = \Delta_p u\quad\textnormal{in }(0,\infty)\times \R^d,
\end{align}
where $d\in\N$ and $p\in (1,\infty)$ are subject to the standing assumptions \eqref{Eq:ranges} below and $\Delta_p$ designates the $p$-Laplacian, i.e.~the divergence of the $p$-gradient $\smash{\vert\nabla \cdot\vert^{p-2}\,\nabla\,\cdot}$\,. Our results equally cover the \emph{Leibenson process} of Barbu--Grube--Rehmeier--Röckner \cite{barbu-grube-rehmeier-rockner2025+}, associated with the Leibenson equation \cite{leibenson1945-general,leibenson1945-turbulent} 
\begin{align}\label{Eq:Leibenson}
\frac{\partial}{\partial t}u = \Delta_p u^m\quad\textnormal{in }(0,\infty)\times \R^d,
\end{align}
where $m>0$ is again subject to \eqref{Eq:ranges}; evidently, \eqref{Eq:Leibenson} generalizes \eqref{Eq:pLaplaceequation}.

The common protagonist of our work is the \emph{radial process}: the distance of the process to a fixed center. Already independently of any geometric context, such processes constitute a fundamental tool of stochastic analysis: they encode exit, hitting, and occupation times as well as maximal and moment estimates. They also underlie nonexplosion and transience criteria. Our contributions, outlined in \cref{Sub:MainRes,Sub:ApplIntro}, are four-fold. 

\begin{itemize}
    \item We establish explicit Tanaka--Meyer semi\-martin\-gale decompositions for the radial parts of these processes, relative to an arbitrary center and dimension. 

    \item We characterize nontriviality of the induced local times --- which reveals notable differences between $p$- and Euclidean Brownian motion --- and  compute their expectations explicitly in terms of the Barenblatt fundamental solution \cite{barenblatt1952,barenblatt2003}. 

    \item We derive sharp exit time estimates and scaling estimates for the self-similarly rescaled radial process.

    \item We establish the (linearized) flow of marginal laws of these processes acts iso\-metrically on Dirac masses with respect to every Wasserstein distance. We deduce a sharp gradient estimate for the dual flow by Kuwada's duality argument \cite{kuwada2010}.
\end{itemize}


\subsection{Nonlinear Markov processes}\label{Sub:pBMIntro} Classical Markov processes describe the Lagrangian dynamics of \emph{linear} Fokker--Planck equations (FPEs); for a comprehensive account, cf.~e.g. Bogachev--Krylov--Röckner--Shaposhnikov \cite{bogachev-krylov-rockner-shaposhnikov2015}. The prime example is the correspondence of Brownian motion and the heat equation --- the linear cousin of \eqref{Eq:pLaplaceequation} --- pioneered in the Euclidean setting by Einstein, von Smoluchowski, and Wiener, and extended far beyond it since, e.g.~to Riemannian manifolds (Eells--Elworthy \cite{elworthy1982} and Malliavin \cite{malliavin1978}) and to infinitesimally Hilbertian metric measure spaces (Ambrosio--Gigli--Savaré \cite{ambrosio-gigli-savare2014-riemannian}).

\emph{Nonlinear Markov processes}, their nonlinear counterparts, were envisioned by McKean \cite{mckean1966}. They have recently been given a systematic treatment by Rehmeier--Röckner \cite{rehmeier-rockner2022+}. Conceptually, they describe the Lagrangian dynamics of \emph{nonlinear FPEs}; we refer to e.g. Barbu--Röckner \cite{barbu-rockner2024} for background on the latter. Nonlinear Markov processes are families of laws whose transition mechanism depends not only on the current state, but also on the current distribution; as such, they are typically realized by McKean--Vlasov SDEs. To associate such a process to a nonlinear PDE, the principal task is to rewrite the PDE as a nonlinear FPE, which in turn dictates the correct McKean--Vlasov dynamics. Besides the $p$-Laplace and Leibenson equations at the heart of our paper, this program has been carried out for the porous medium equation (Rehmeier--Röckner \cite{rehmeier-rockner2022+}), the 2D vorticity Euler equation (Rehmeier--Romito \cite{rehmeier-romito2026}), the Navier--Stokes equations (Barbu--Röckner--Zhang \cite{barbu-rockner-zhang2025}), and Brownian motion on Minkowski normed spaces (Ohta--Rehmeier--Suzuki \cite{ohta-rehmeier-suzuki2026+}). Beyond the identification of the SDE, these works prove the nonlinear Markov property along with further properties pertaining, e.g.,~to pathwise uniqueness.

Just like for classical Brownian motion, for the nonlinear PDEs \eqref{Eq:pLaplaceequation} and \eqref{Eq:Leibenson}, natural candidates for the marginal densities are the Barenblatt fundamental solutions $w^x$ \cite{barenblatt1952,barenblatt2003}, where $\smash{x\in \R^d}$, cf.~\cref{Def:Barenblatt}; following the above line of thought, Barbu--Rehmeier--Röckner \cite{barbu-rehmeier-rockner2026} and Barbu--Grube--Rehmeier--Röckner \cite{barbu-grube-rehmeier-rockner2025+} identified the McKean--Vlasov SDE, viz.~\eqref{Eq:OriginalMCV}, whose time $t$ marginals are precisely $\smash{w^x(t,\cdot)\,\Leb^d}$ for every $t>0$. 

Under our standing assumptions \eqref{Eq:ranges}, \cite{barbu-rehmeier-rockner2026,barbu-grube-rehmeier-rockner2025+} show probabilistically weak solutions to this SDE exist, cf.~\cref{Th:ExistencePWS}; under stronger hypotheses on $d$, $p$, and $m$, these are probabilistically strong, pathwise unique, and their laws form a time-homogeneous nonlinear Markov process. However, we stress that our work merely requires probabilistically weak solutions; neither the nonlinear Markov property nor pathwise uniqueness will be used. This robustness is not merely a matter of convenience: as recently shown by Abedi--Bechtold--Rehmeier \cite{abedi-bechtold-rehmeier2026+}, a nonlinear PDE may admit several distinct nonlinear Markov processes, which are not determined by their one-dimensional marginal laws. 

These processes exhibit distinctly nonlinear features their linear relatives do not have: their marginals are compactly supported --- reflecting the finite speed of propagation of \eqref{Eq:pLaplaceequation} and \eqref{Eq:Leibenson} we will stipulate through \eqref{Eq:ranges} --- and their diffusivity degenerates both at the starting point and at the boundary of the Barenblatt support. Any pathwise theory has to confront these degeneracies; as we will see, they are also responsible for genuinely new phenomena pertaining to the local times induced by the associated radial processes. 

Local times have appeared in connection with McKean--Vlasov dynamics in different guises. Agram--Øksendal \cite{agram-oksendal2025} represent the local time of a one-dimensional distribution dependent process through the Donsker delta function, whereas in the particle systems of Baker--Hambly--Jettkant \cite{baker-hambly-jettkant2025+} local times constitute the interaction mechanism itself; Tanaka--Meyer formulas for interacting measure-valued branching diffusions are due to Dawson--Vaillancourt--Wang \cite{dawson-vaillancourt-wang2023}. To our knowledge, the local time of the radial part of a nonlinear Markov process has not been studied before.

\subsection{Classical Brownian motion in Riemannian geometry}\label{Sub:ClassBM} The interplay of analysis and probability through classical Brownian motion flourishes particularly on a Riemannian geometric background. Mediated  prominently by lower bounds on the Ricci curvature, it has been a driving force of geometric and stochastic analysis over the last decades, cf. e.g. Hsu \cite{hsu2002} and Wang \cite{wang2014-analysis} for comprehensive accounts. 

A mechanism from this circle of ideas which is paradigmatic for our work is Kendall's semimartingale decomposition of radial processes \cite{kendall1987}. Let $(\mms,g)$ denote a complete Riemannian manifold of dimension $n\geq 2$. Let $\Delta$ be the associated Laplace--Beltrami operator. Fix a center $o\in \mms$. Let $\rO$ be the distance function to $o$. Then, letting $X$ be Brownian motion on $\mms$ with generator $\Delta$ starting at $o$,
\begin{align}\label{Eq:KendallBM}
    \rmd \rO(X_t) = \sqrt{2}\,\rmd B_t + \Delta\rO(X_t)\d t - \rmd \Lambda_t
\end{align}
holds a.s.~before explosion, where 

\begin{itemize}
    \item $B$ is a standard Brownian motion on $\R$ and 

    \item $\Lambda$ is a nondecreasing process increasing only when $X$ hits the cut locus of $o$. 
\end{itemize}
Two structural features of \eqref{Eq:KendallBM} deserve emphasis. 

\begin{enumerate}[label=\arabic*\textcolor{black}{.}]
    \item\label{La:11111} The term $-\rmd \Lambda_t$ is \emph{nonpositive} and may simply be discarded when one aims at \emph{upper} bounds on the radial process. Heuristically, this sign is an artefact of a \emph{concave} kink of $\smash{\rO}$ at the cut locus, being a minimum of smooth distance branches; it forces the process to reflect inwards, leading to a decrease of the distance of $X_t$ to $o$. 

    \item\label{La:222222222222222222222222222222222222} The drift is governed by $\Delta\rO$, which encodes \emph{curvature}: by Calabi's Laplacian comparison theorem \cite{calabi1958}, a lower Ricci curvature bound forces an upper bound on $\Delta \rO$ (in a suitable sense) by the radial Laplacian of the corresponding constant curvature model. For instance, nonnegative Ricci curvature entails
\begin{align}\label{Eq:CalabiComp}
    \Delta\rO\leq \frac{n-1}{\rO},
\end{align}
which is sharp in the Euclidean model space outside the origin.
\end{enumerate}

Combining both features, the radial process is dominated by a Bessel-type process of the model space. This single mechanism underlies exit and hitting time estimates, moment bounds, nonexplosion criteria, etc. Adaptations of \eqref{Eq:KendallBM} have since been established, e.g., for Riemannian manifolds with boundary (Bass--Hsu \cite{bass-hsu1990}), RCD spaces (Kuwada--Kuwae \cite{kuwada-kuwae2019}), and sub-Riemannian manifolds (Baudoin--Grong--Kuwada--Neel--Thalmeier \cite{baudoin-grong-kuwada-neel-thalmeier2020}). 

Close relatives of the deep connection of Brownian motion and Ricci curvature include coupling by parallel transport as pioneered by Kendall \cite{kendall1986} and Cranston \cite{cranston1991} and the characterization of lower Ricci bounds through Wasserstein contractivity of the heat flow by von Renesse--Sturm \cite{von-renesse-sturm2005} and, dually via Kuwada's duality \cite{kuwada2010}, through \smash{Bakry--Émery} gradient estimates \cite{bakry-emery1985-diffusions}. 

All these are instances of one well-tried motto: a lower Ricci curvature bound forces geometric and probabilistic quantities to be bounded from above by their counterparts in the comparison model. Given the wealth of interactions for classical Brownian motion, it is natural to ask whether $p$-Brownian motion supports a comparably strong geometric theory on Riemannian manifolds. We believe it does and the present article provides the entry point by settling the Euclidean case for the more general Leibenson process: a thorough analysis of its induced radial processes, comprising explicit semimartingale formulas à la Tanaka--Meyer alongside sharp exit time and self-similar scaling estimates.

\subsection{Main results}\label{Sub:MainRes} Throughout we fix numbers $d\in\N$, $p\in (1,\infty)$, and $m>0$ subject to the slow diffusion regime \eqref{Eq:ranges}. Set $\beta := p+d(m(p-1)-1)$ as in \eqref{Eq:betagammakappa}. For $\smash{x\in\R^d}$, let $\smash{w^x}$ denote the Barenblatt solution from \cref{Def:Barenblatt}. For $t>0$, the support of $\smash{w^x(t,\cdot)}$ is precisely $\smash{\overline{B}_{R(t)}(x)}$, where $\smash{R(t):=\rho\,t^{1/\beta}}$ is the \emph{Barenblatt radius} with an explicit constant $\rho>0$ that depends only on $d$, $p$, and $m$, cf.~\eqref{Eq:Rdef}. Let $X$ constitute a probabilistically weak solution to the  McKean--Vlasov SDE \eqref{Eq:OriginalMCV} governing the Leibenson process starting at $\smash{x\in\R^d}$. It is modeled on a suitable filtered probability space $\smash{(\Omega,\scrF,\scrF_\bullet,\PPP^x)}$. Fix a center $\smash{o\in\R^d}$ inducing the distance function $\smash{\rO\colon \R^d\to\R_+}$ given by 
\begin{align*}
\rO(y) :=\vert y-o\vert.    
\end{align*}

Our first main result summarizes \cref{Th:CenteredRadial,Th:Summary} in the centered case, i.e.~when the starting point $x$ and the center $o$ coincide. 

\begin{theorem}[Tanaka--Meyer formula for centered radial process]\label{Th:IntroTanaka} Assume \eqref{Eq:ranges} and let $\smash{o\in\R^d}$. Then there exists a standard $\smash{\scrF_\bullet}$-Brownian motion $B$ on $\R$ such that
\begin{align}\label{Eq:Blabla}
\begin{split}
\rmd \rO(X_t) &= \sqrt{2m^{p-1}}\,\big\vert\nabla w^o(t,X_t)\big\vert^{(p-2)/2}\,w^o(t,X_t)^{(m-1)(p-1)/2}\d B_t\\
    &\qquad\qquad + m^{p-1}\div\!\big[\big\vert\nabla w^o(t,X_t)\big\vert^{p-2}\,w^o(t,X_t)^{(m-1)(p-1)}\,\nabla\rO(X_t)\big]\d t\quad\PPP^o\textnormal{\textit{-a.s.}}
\end{split}
\end{align}
\end{theorem}

The drift has a geometrically interesting form we comment on further in \cref{Sub:Outlook}.

This and further of our results are compared to Euclidean Brownian motion --- which is \emph{not} covered by our standing hypotheses \eqref{Eq:ranges} --- in \cref{Re:BMLocal}. In a nutshell, in the centered situation, \cref{Th:IntroTanaka} implies $X$ has no local time at $o$ in \emph{any} dimension. This property is \emph{different} from Euclidean Brownian motion precisely in one dimension; see also the paragraph after \cref{Th:IntroLocal}.

In the uncentered case, we show a similar Tanaka--Meyer formula in \cref{Th:RadialOne,Th:RadialTwo,Th:Summary}; it turns out to contain an additional \emph{nontrivial} local time at $o$ exactly in one dimension. This is \emph{analogous} to Euclidean Brownian motion.

The following is an immediate consequence of our Tanaka--Meyer formulas.

\begin{corollary}[Semimartingale property of radial process] For every $\smash{x\in\R^d}$, the radial process $\rO\circ X$ is a semimartingale under $\PPP^{x}$.
\end{corollary}

Furthermore, using the well-developed one-dimensional theory of local times, cf.~e.g. Revuz--Yor \cite{revuz-yor1999}, we show several fundamental properties of the local time $L^o$ of $X$ at $o$ in the nontrivial case. First, it increases only when the process $X$ hits $o$, cf.~\cref{Pr:JumpsLocalTime}. Second, we show an occupation time identity $\PPP^x$-a.s.~for every $t\in\R_+$, cf.~\cref{Th:Modif}: 
\begin{align}\label{Eq:IntroOccId}
    L_t^o = \lim_{\varepsilon\to 0^+}\frac{2m^{p-1}}{\varepsilon}\int_{[0,t]} 1_{\{o\leq X_s\leq o+\varepsilon\}}\,\big\vert\nabla w^x(s,X_s)\big\vert^{p-2}\,w^x(s,X_s)^{(m-1)(p-1)}\d s.
\end{align}

Third, we compute its expectation precisely in \cref{Th:Modif} and \eqref{Eq:Contrast}. 

\begin{theorem}[Expectation of nontrivial local time]\label{Th:IntroLocal} Assume \eqref{Eq:ranges} with $d=1$. Let $\smash{x\in\R}$. Then for every $t\in\R_+$,
    \begin{align}\label{Eq:IntroExpectation}
        \EEE^x\big[L_t^o\big] = \begin{cases} \displaystyle 2m^{p-1}\int_{[0,t]} \big\vert \nabla w^x(s,o)\big\vert^{p-2}\,w^x(s,o)^{(m-1)(p-1)+1}\d s & \textnormal{\textit{if} } x\neq o,\\
        0 & \textnormal{\textit{otherwise}}.
        \end{cases}
    \end{align}
\end{theorem}

The right-hand side admits precise asymptotics, cf.~\cref{Cor:ExpLocalTime}. 

Heuristically, \eqref{Eq:IntroOccId} shows the local time of the radial process weighs the occupation of $o$ by $X$ by a diffusivity that behaves like $\smash{r^{(p-2)/(p-1)}}$, cf.~\cref{Le:aU-bounds-m}. The exponent is positive since \eqref{Eq:ranges} forces $p>2$ in dimension one. In particular, formally sending $x\to o$ deduces the second case from \eqref{Eq:IntroExpectation} from the first. In general, the identity $\smash{\EEE^x[L_t^o]=0}$ in the deterministic regime $\smash{t \in [0,R^{-1}(\vert o-x\vert))}$ implied by \eqref{Eq:IntroExpectation} mirrors the finite speed of propagation: before the support of $\smash{w^x(t,\cdot)}$ reaches $o$, the process cannot have hit $o$.

\subsection{Estimates for the radial process}\label{Sub:ApplIntro} We also obtain exit time and scaling estimates in \cref{Sub:Exittime}. Let $\tau_R$ be the first exit time of $X$ from $B_R(o)$, where $R>0$. In \cref{Pr:DetExit,Th:Exittime}  we prove the deterministic lower bound $\smash{\tau_R\geq \rho^{-\beta}(R-\vert x-o\vert)^\beta}$ $\smash{\PPP^x}$-a.s.~and the complementary survival bound $\smash{\PPP^o[\tau_R>t]\lesssim \min\{1, R^d\,t^{-d/\beta}\}}$ for every $t>0$. This expresses the finite propagation speed without any counterpart for Euclidean Brownian motion. Together these results pin down a natural exit time scale for $\tau_R$, which we quantify by the moment asymptotics of \cref{Th:Expexittime}. These results rest mainly on the explicit Barenblatt radius \eqref{Eq:Rrho}. We then turn to the self-similarly rescaled radial process $Y$ defined by $\smash{Y_t=\rO(X_t)/t^{1/\beta}}$. Its law under $\smash{\PPP^o}$ turns out not to depend on the time $t>0$, cf.~\cref{Pr:LawResc}. Here our  semimartingale formulas enter: bounds in expectation\footnote{The degeneracy of the diffusivity at the center rules out deterministic ones.} for the coefficients of our Tanaka--Meyer formula from \cref{Le:Coeffq} yield a stopped moment estimate, cf.~\cref{Th:Stopped}, and in turn an exit estimate over moving self-similar boundaries, cf.~\cref{Cor:MovBoundary}, controlling excursions of $Y$ over levels below the Barenblatt boundary.

Finally, in \cref{Sub:Appendix} we investigate the \emph{Barenblatt flow} $\sfh_\bullet$, which is the flow of marginal laws of $p$-Brownian motion, and its (linear) dual $\sfp_\bullet$ acting on functions. We show for every $t\in\R_+$, $\sfh_t$ acts isometrically on Dirac masses in \emph{every} Wasserstein distance $W_q$, $q\in[1,\infty]$, cf. \cref{Pr:WasserIso}, and deduce a sharp gradient estimate in \cref{Th:DualBarFlo}. The proofs deliberately proceed by coupling and by Kuwada's duality argument \cite{kuwada2010} rather than by explicit computation, anticipating their generalization beyond the Euclidean setting.

\subsection{Outlook}\label{Sub:Outlook} Inspired by the classical linear theory recalled in \cref{Sub:ClassBM}, it would be interesting to extend the construction of $p$-Brownian motion of Barbu--Rehmeier--Röckner \cite{barbu-rehmeier-rockner2026} --- and, likewise, of the Leibenson process of Barbu--Grube--Rehmeier--Röckner \cite{barbu-grube-rehmeier-rockner2025+} --- to Riemannian manifolds, and to explore the connections of the resulting processes to curvature. In the remainder of this subsection, we indicate why such a theory seems to us to be within reach; we hope the present work provides a useful point of departure for it.

Many deep connections of the $p$-Laplacian and geometry are known. Analytic properties of the Leibenson equation \eqref{Eq:Leibenson} on Riemannian manifolds have recently been studied e.g.~by Sürig \cite{surig2024+,surig2025+-gradient,surig2026+-existence,surig2026+-long} and Grigor'yan--Sürig \cite{grigoryan-surig2024,grigoryan-surig2025} with Sun \cite{grigoryan-sun-surig2026+}. A nonlinear Bochner formula and spectral gap estimates are established by Matei \cite{matei2000}, Valtorta \cite{valtorta2012}, and Naber--Valtorta \cite{naber-valtorta2014}. Ohta \cite{ohta2014} has shown a general comparison theorem. Moreover, Calabi's classical comparison theorem for the Laplace--Beltrami operator directly transfers to the $p$-Laplacian without any change, since the gradient of a distance function $\rO$ has magnitude one (and thus $\Delta_p\rO = \Delta \rO$) a.e.

A probabilistic counterpart to these efforts appears natural, and our main results suggest what its starting point could be. Our Tanaka--Meyer formulas have a radial structure we frequently exploit, cf.~\cref{Sub:RadialRepSec}: for instance, the drift from \eqref{Eq:Blabla} reads
\begin{align*}
    a(t,\rO(X_t))\,\frac{d-1}{\rO(X_t)}\d t + b(t,\rO(X_t))\d t,
\end{align*}
where $\smash{a\colon (0,\infty)\times\R_+\to\R_+}$ and $\smash{b\colon (0,\infty)\times\R_+\to\R}$ are easily seen to satisfy
\begin{align}
    a(t,\rO(X_t)) &= m^{p-1}\,\big\vert\nabla w^o(t,X_t)\big\vert^{p-2}\,w^o(t,X_t)^{(m-1)(p-1)},\\
    b(t,\rO(X_t)) &= m^{p-1}\, \nabla\big[\big\vert\nabla w^o(t,X_t)\big\vert^{p-2}\,w^o(t,X_t)^{(m-1)(p-1)}\big] \cdot \nabla \rO(X_t)\quad\PPP^o\textnormal{-a.s.};
\end{align}
in particular, $b$ is the radial derivative of $a$, cf.~\eqref{Eq:a_rdef}. We observe the drift contains the \emph{linear} Euclidean Laplacian $\smash{\Delta\rO}$ of the distance function $\rO$ --- making Ricci curvature enter Kendall's decomposition \eqref{Eq:KendallBM} by \eqref{Eq:CalabiComp} --- multiplied by a \emph{non\-negative} diffusivity. This nonnegativity clearly persists on a more geometric background for any reasonable replacement of the Baren\-blatt solution. Our centered radial process forms the nonlinear analog of the Bessel-type process dominating \eqref{Eq:KendallBM} under \eqref{Eq:CalabiComp}. In turn, in a Riemannian framework our martingale methods from \cref{Sub:Exittime} could shed new light on the parabolic $p$-Laplace equation or the Leibenson equation, whose solutions are typically not explicit --- unlike the Euclidean Barenblatt solutions.

This is a conceptual connection of our results with structural feature \ref{La:222222222222222222222222222222222222} of \eqref{Eq:KendallBM} from \cref{Sub:ClassBM}. The next remark comments on structural feature \ref{La:11111}

\begin{remark}[About Kendall's local time at the cut locus] The local times in \eqref{Eq:KendallBM} and our work have \emph{opposite} signs: the former enters with a negative, the latter with a positive sign. These objects simply describe something different. The Tanaka--Meyer formula picks up the singular part of the distributional generator of $\rO$, together with its sign. Kendall's local time lives on the cut locus, where $\rO$ exhibits a \emph{concave} kink, whence the negative sign. Our local time lives at the center $o$, where $\rO$ exhibits a \emph{convex} kink, whence the positive sign. The same happens for classical Brownian motion, cf. \cref{Re:BMLocal}.
\end{remark}

Of course, in our work we do not have to deal with cut loci. We expect them to enter Riemannian generalizations of our Tanaka--Meyer formulas, similarly to \eqref{Eq:KendallBM}.

Further geometric applications suggest themselves, such as nonlinear characterizations of Ricci curvature in the spirit of von Renesse--Sturm \cites{von-renesse-sturm2005}.

In the classical theory of Dirichlet forms, a local time is well-known to arise from the positive continuous additive functional induced by a smooth measure on its ``support'' by Revuz correspondence, cf.~e.g.~Chen--Fuku\-shima \cite{chen-fukushima2012}. It would also be interesting to study our local times abstractly (by developing nonlinear analogs of the aforementioned terms) from the perspective of \emph{nonlinear Dirichlet functionals}. The latter have recently attracted increasing interest, cf.~e.g.~Brigati--Dello Schiavo \cite{brigati-dello-schiavo2025+} and its references.

\section{Preliminaries}
Throughout, we fix $d\in\N$, $p\in (1,\infty)$, and $m > 0$, and assume
\begin{align}\label{Eq:ranges}
\begin{split}
    m(p-1)&>1,\\
dp &>d+1.
\end{split}
\end{align}
Among other things, this implies the existence of suitable solutions to the McKean--Vlasov SDEs we consider in this paper, cf.~\cref{Th:ExistencePWS}. On the other hand, the first condition from \eqref{Eq:ranges} excludes Euclidean Brownian motion (the borderline case $m=1$ and $p=2$). 

Let $\cdot$ and $:$ denote the standard scalar products in $\smash{\R^d}$ and $\smash{\R^{d\times d}}$, respectively. 

Given two scalar expressions $A$ and $B$, we will write $A\lesssim B$ if there is a constant $c>0$ that only depends on the three parameters $p$, $m$, and $d$ (that are fixed throughout the paper) with $A\leq c\,B$. If we allow $c$ to also depend on a specified parameter $x$, we write $A\lesssim_x B$. Lastly, we will write $A\asymp B$ if there exists a constant $c>0$ depending only on $p$, $m$, and $d$ such that $A/c \leq B \leq c\,A$.
\subsection{Barenblatt solution} Let $\omega_{d-1}$ denote the area of the sphere $\S^{d-1}$ in $\R^d$, with the  interpretation $\omega_{0}= 2$. We define
\begin{align}\label{Eq:betagammakappa}
\begin{split}
    \beta &:= p+d(m(p-1)-1),\\
    \gamma &:= \frac{p-1}{m(p-1)-1},\\
    \kappa &:=\frac{m(p-1)-1}{mp}\,\beta^{-1/(p-1)}
\end{split}
\end{align}
and note that $\beta>p$ and $\gamma,\kappa>0$ under our hypothesis $m(p-1)>1$. Define the function $S\colon (0,\infty)\times \R_+\to \R$ by
\begin{align}\label{Eq:f}
    S(t,r):= C-\kappa\big[t^{-1/\beta}\,r\big]^{p/(p-1)},
\end{align}
and the induced Barenblatt profile $U\colon (0,\infty)\times\R_+\to\R_+$ by
\begin{align}\label{Eq:U}
U(t,r):=t^{-d/\beta}\,S(t,r)_+^\gamma,
\end{align}
where $C>0$ is unique (depending only on $d$, $p$, and $m$) such that for every $t>0$,
\begin{align}\label{Eq:Normal}
    \omega_{d-1}\int_{\R_+} U(t,r)\,r^{d-1}\d r=1.
\end{align}
Given $t>0$, the support of $U(t,\cdot)$ is precisely $[0,R(t)]$, where
\begin{align}\label{Eq:Rdef}
    R(t):=\Big[\frac{C}{\kappa}\Big]^{1-1/p}t^{1/\beta}.
\end{align}
The inverse of $R\colon\R_+\to\R_+$ is given by
\begin{align}\label{Eq:R-1def}
    R^{-1}(s) :=\Big[\frac{\kappa}{C}\Big]^{\beta(1-1/p)}s^\beta.
\end{align}

\begin{definition}[Barenblatt solution]\label{Def:Barenblatt} Given $x\in\R^d$, the \emph{Barenblatt solution} $w^x\colon (0,\infty)\times\R^d\to \R_+$ to the Leibenson equation \eqref{Eq:Leibenson} is defined by
\begin{align*}
    w^x(t,y) := U(t,\vert y-x\vert).
\end{align*}
\end{definition}
The normalization property \eqref{Eq:Normal} ensures for every $t>0$ and every $\smash{x\in\R^d}$, $w^x$ is a probability density with respect to $\smash{\Leb^d}$. Furthermore, it is a fundamental solution to the Leibenson equation \eqref{Eq:Leibenson}, in that it is a weak solution obeying $w^x(t,\cdot)\to \delta_x$ distributionally as $\smash{t\to 0^+}$. On the other hand, it is \emph{not} a fundamental solution in the sense of being an integral kernel of solutions with general initial data, because \eqref{Eq:Leibenson} is nonlinear. 
By construction, given $t>0$ and $\smash{x\in\R^d}$ the support of $w^x$ is $\smash{\overline{B}_{R(t)}(x)}$, where $R$ is from \eqref{Eq:Rdef}.
\subsection{Associated McKean--Vlasov SDE} Barbu--Rehmeier--Röckner \cite{barbu-rehmeier-rockner2026} have identified a McKean--Vlasov SDE to govern the Lagrangian dynamics of the parabolic $p$-Laplace equation \eqref{Eq:pLaplaceequation}. This was later extended to the more general Leibenson equation \eqref{Eq:Leibenson} by Barbu--Grube--Rehmeier--Röckner \cite{barbu-grube-rehmeier-rockner2025+}: fixing an initial point $\smash{x\in\R^d}$ and considering the Barenblatt solution $w^x$ from \cref{Def:Barenblatt}, their proposed McKean--Vlasov SDE reads
\begin{align}\label{Eq:OriginalMCV}
\begin{split}
    \rmd X_t &= \sqrt{2m^{p-1}}\,\big\vert\nabla w^x(t,X_t)\big\vert^{(p-2)/2}\,w^x(t,X_t)^{(m-1)(p-1)/2}\d W_t\\
    &\qquad\qquad + m^{p-1}\,\nabla\big[\big\vert \nabla w^x(t,X_t)\big\vert^{p-2}\,w^x(t,X_t)^{(m-1)(p-1)}\big]\d t,\\
    \Law\, X_t &= w^x(t,\cdot)\,\Leb^d,
\end{split}
\end{align}
where $W$ is a $d$-dimensional standard Brownian motion on a stochastic basis, where the latter is a quadruple $(\Omega,\scrF,\scrF_\bullet,\PPP^x)$, where $(\Omega,\scrF,\PPP^x)$ is a complete probability space with a right-continuous filtration $\scrF_\bullet$ augmented by $\PPP^x$-null sets; we write $\PPP^x$ instead of $\PPP$ to stress the dependence on the starting point, cf.~\eqref{Eq:Law0}. Barbu--Grube--Rehmeier--Röckner \cite{barbu-grube-rehmeier-rockner2025+}*{Def.~3.1} have defined a notion of solution to \eqref{Eq:OriginalMCV} as follows.
\begin{definition}[Probabilistically weak solution to \eqref{Eq:OriginalMCV}]\label{Def:PWS} Given $\smash{x\in\R^d}$, a \emph{probabilistically weak solution} to \eqref{Eq:OriginalMCV} is an $\scrF_\bullet$-adapted continuous stochastic process $X$ on a stochastic basis $(\Omega,\scrF,\scrF_\bullet,\PPP^x)$ carrying an $\scrF_\bullet$-standard Brownian motion $W$ on $\smash{\R^d}$ such that
\begin{enumerate}[label=\textnormal{\alph*.}]
    \item for every $t>0$,
    \begin{align*}
        \EEE^x\Big[\!\int_{[0,t]}\big\vert\nabla w^x(s,X_s)\big\vert^{p-2}\,w^x(s,X_s)^{(m-1)(p-1)}\d s\Big] &< \infty,\\
        \EEE^x\Big[\!\int_{[0,t]}\big\vert\nabla\big[\big\vert\nabla w^x(s,X_s)\big\vert^{p-2}\,w^x(s,X_s)^{(m-1)(p-1)}\big]\big\vert\d s\Big] &< \infty,
    \end{align*}
    \item $\PPP^x$-a.s., every $t\in\R_+$ satisfies
    \begin{align*}
        X_t&=X_0+\sqrt{2m^{p-1}}\int_{[0,t]} \big\vert \nabla w^x(s,X_s)\big\vert^{(p-2)/2}\,w^x(s,X_s)^{(m-1)(p-1)/2}\d W_s\\
        &\qquad\qquad + m^{p-1}\int_{[0,t]}\nabla\big[\big\vert\nabla w^x(s,X_s)\big\vert^{p-2}\,w^x(s,X_s)^{(m-1)(p-1)}\big]\d s,
    \end{align*}
    \item for every $t>0$,
    \begin{align*}
    \Law\, X_t=w^x(t,\cdot)\,\Leb^d.
    \end{align*}
    \end{enumerate}
\end{definition}
We will always fix the stochastic basis $(\Omega,\scrF,\scrF_\bullet,\PPP^x)$ and the Brownian motion $W$ and briefly call $X$ a probabilistically weak solution to \eqref{Eq:OriginalMCV}. \cref{Def:PWS} implies
\begin{align}\label{Eq:Law0}
    \Law\,X_0=\delta_x.
\end{align}
We thus occasionally say that the above process starts at $x$. 

A general criterion for existence has been established by Barbu--Rehmeier--Röckner \cite{barbu-rehmeier-rockner2026}*{Thm.~4.2, Cor.~4.3} for the parabolic $p$-Laplace PDE \eqref{Eq:pLaplaceequation} as well as by Barbu--Grube--Rehmeier--Röckner \cite{barbu-grube-rehmeier-rockner2025+}*{Thm.~4.4, Cor.~4.5} for the Leibenson equation \eqref{Eq:Leibenson}.
\begin{theorem}[Existence of probabilistically weak solutions to \eqref{Eq:OriginalMCV}]\label{Th:ExistencePWS} Our standing assumptions \eqref{Eq:ranges} imply that for every $\smash{x\in\R^d}$, there is a probabilistically weak solution $X$ to the McKean--Vlasov SDE \eqref{Eq:OriginalMCV}.
\end{theorem}
As shown by Barbu--Rehmeier--Röckner \cite{barbu-rehmeier-rockner2026}*{Thm.~3.3} and Barbu--Grube--Rehmeier--Röckner \cite{barbu-grube-rehmeier-rockner2025+}*{Prop.~3.2, Cor.~3.4}, through the marginal densities every probabilistically weak solution to \eqref{Eq:OriginalMCV} corresponds to a weak solution of \eqref{Eq:Leibenson}, which rewrites as the following nonlinear Fokker--Planck equation:
\begin{align*}
    \frac{\partial}{\partial t}w^x &= m^{p-1}\,\Delta\big[\big\vert\nabla w^x\big\vert^{p-2}\,(w^x)^{(m-1)(p-1)}\,w^x\big]\\
    &\qquad\qquad - m^{p-1}\div\!\big[\nabla\big[\big\vert \nabla w^x\big\vert^{p-2}\,(w^x)^{(m-1)(p-1)}\big]\,w^x\big]\quad\textnormal{on }(0,\infty)\times\R^d.
\end{align*}
For \cref{Sec:RadialTanakaMeyer}, it will be convenient to rewrite this again as a \emph{quasilinear} parabolic PDE by
\begin{align*}
    \frac{\partial}{\partial t}w^x = \sfL_{p,m}^xw^x\quad\textnormal{on }(0,\infty)\times\R^d,
\end{align*}
where for a function $\smash{u\colon (0,\infty)\times\R^d\to\R}$, we set
\begin{align}\label{Eq:Generator}
\begin{split}
    \sfL_{p,m}^xu&:=m^{p-1}\,\Delta\big[\big\vert\nabla w^x\big\vert^{p-2}\,(w^x)^{(m-1)(p-1)}\,u\big]\\
    &\qquad\qquad - m^{p-1}\div\!\big[\nabla\big[\big\vert \nabla w^x\big\vert^{p-2}\,(w^x)^{(m-1)(p-1)}\big]\,u\big]
    \end{split}
\end{align}
whenever the right-hand side makes sense.
Two further central results of \cite{barbu-rehmeier-rockner2026,barbu-grube-rehmeier-rockner2025+} are that under stronger hypotheses on $d$, $p$, and $m$,
\begin{itemize}
    \item the solution laws of the probabilistically weak solutions to \eqref{Eq:OriginalMCV} constitute a time-homogeneous nonlinear Markov process in the sense of McKean \cite{mckean1966} as studied by Rehmeier--Röckner \cite{rehmeier-rockner2022+}, and
    \item the probabilistically weak solutions are in fact probabilistically strong, i.e.~measurable functions of the driving Brownian motion and the initial condition; moreover, pathwise uniqueness holds for the McKean--Vlasov SDE \eqref{Eq:OriginalMCV}.
\end{itemize}
We refer to Barbu--Rehmeier--Röckner \cite{barbu-rehmeier-rockner2026}*{§§5--6} and Barbu--Grube--Rehmeier--Röckner \cite{barbu-grube-rehmeier-rockner2025+}*{§§5--6} for details. We will not use these properties in our work.

\subsection{Radial representation of the associated McKean--Vlasov SDE}\label{Sub:RadialRepSec} It will be convenient to employ the radial structure of the coefficients of \eqref{Eq:OriginalMCV} and rewrite this McKean--Vlasov SDE in terms of the Baren\-blatt profile $U$ from \eqref{Eq:U}. Recall the function $R$ from \eqref{Eq:Rdef} and define $\smash{a\colon (0,\infty)\times\R_+\to\R_+}$ by
\begin{align}\label{Eq:a}
    a(t,r) &:= \begin{cases} m^{p-1}\,\big\vert  U_r(t,r)\big\vert^{p-2}\,U(t,r)^{(m-1)(p-1)} & \textnormal{if }r<R(t),\\
    0 & \textnormal{otherwise}
    \end{cases}
\end{align}
where we abbreviate, given $t>0$ and $r\in (0,R(t))$,
\begin{align}\label{Eq:a_rdef}
\begin{split}
U_r(t,r) &:= \frac{\partial}{\partial r}U(t,r),\\
    a_r(t,r) &:=\frac{\partial}{\partial r}a(t,r).
    \end{split}
\end{align}
Given $\smash{z\in\R^d}$, we also define
\begin{align}\label{Eq:sgn}
    \sgn z := \begin{cases}
        \displaystyle\frac{z}{\vert z\vert} & \textnormal{if }z\neq 0,\\
        0 & \textnormal{otherwise}.
    \end{cases}
\end{align}
Finally, consider the McKean--Vlasov SDE
\begin{align}\label{Eq:ClosedMcKeanVlasov}
\begin{split}
    \rmd X_t &=\sqrt{2a(t,\vert X_t-x\vert)}\d W_t +  a_r(t,\vert X_t-x\vert)\sgn(X_t-x)\d t,\\
    \Law\,X_t& = U(t, \vert\cdot- \,x\vert)\, \Leb^d.
\end{split}
\end{align}

\begin{definition}[Probabilistically weak solution to \eqref{Eq:ClosedMcKeanVlasov}] Given $\smash{x\in\R^d}$, a  \emph{probabilistically weak solution} to \eqref{Eq:ClosedMcKeanVlasov} is an $\scrF_\bullet$-adapted continuous stochastic process $X$ on a stochastic basis $(\Omega,\scrF,\scrF_\bullet,\PPP^x)$ carrying an $\scrF_\bullet$-standard Brownian motion $W$ on $\smash{\R^d}$ such that
\begin{enumerate}[label=\textnormal{\alph*.}]
    \item for every $t>0$,
    \begin{align}\label{Eq:Integrability}
\begin{split}
    \EEE^{x}\Big[\!\int_{[0,t]} a(s,\vert X_s-x\vert)\d s\Big]&<\infty,\\
    \EEE^{x}\Big[\!\int_{[0,t]}\big\vert a_r(s,\vert X_s-x\vert)\big\vert\d s\Big]&<\infty,
    \end{split}
\end{align}
\item $\PPP^x$-a.s., every $t\in\R_+$ satisfies
\begin{align*}
    X_t = X_0 + \sqrt{2}\int_{[0,t]}\sqrt{a(s,\vert X_s-x\vert)}\d W_s + \int_{[0,t]} a_r(s,\vert X_s-x\vert)\sgn(X_s-x)\d s,
\end{align*}
\item for every $t>0$,
\begin{align*}
    \Law\,X_t = U(t,\vert \cdot-\,x\vert)\d\Leb^d.
\end{align*}
\end{enumerate}
\end{definition}

By the chain rule, the following is clear.

\begin{proposition}[Identification of probabilistically weak solutions to \eqref{Eq:OriginalMCV} and \eqref{Eq:ClosedMcKeanVlasov}] The probabilistically weak solutions to \eqref{Eq:OriginalMCV} and \eqref{Eq:ClosedMcKeanVlasov} coincide.
\end{proposition}

The radial formulation \eqref{Eq:ClosedMcKeanVlasov} of \eqref{Eq:OriginalMCV} makes the following lemma clear. To state it, we define the translation $\smash{T_v\colon\R^d\to\R^d}$ by a vector $\smash{v\in\R^d}$ by
\begin{align}\label{Eq:Translation}
T_v(x):=x+v.
\end{align}

\begin{lemma}[Translation]\label{Le:Translation} For $x\in\R^d$, let $X$ be a probabilistically weak solution to  \eqref{Eq:ClosedMcKeanVlasov} starting at $x$. Then for every $y\in\R^d$, the process $Y:=X-x+y$ is a probabilistically weak solution to \eqref{Eq:ClosedMcKeanVlasov} starting at $y$; in particular, every $t\in\R_+$ satisfies
\begin{align}\label{Eq:Law}
    \Law\,Y_t = (T_{y-x})_\push\Law\,X_t.
\end{align}
\end{lemma}
Geometrically, the translation from the previous lemma corresponds to coupling the laws in question by parallel transport, a well-known coupling technique from stochastic differential geometry pioneered by Kendall \cite{kendall1986} and Cranston \cite{cranston1991}.

We conclude this subsection by stating fundamental estimates for the coefficients of \eqref{Eq:ClosedMcKeanVlasov} that we will use frequently.
\begin{lemma}[Fundamental estimates]\label{Le:aU-bounds-m} Consider the functions $a$, $S$, $U$, and $R$ from \eqref{Eq:a}, \eqref{Eq:f}, \eqref{Eq:U}, and \eqref{Eq:Rdef}, respectively.
Then the following statements hold.
\begin{enumerate}[label=\textnormal{\textcolor{black}{(}\roman*\textcolor{black}{)}}]
    \item\label{La:Eins} For every $t>0$ and every $r\in (0,R(t))$,
    \begin{align}
        a(t,r) &\asymp t^{-1+p/(p-1)\beta}\,r^{(p-2)/(p-1)}\,S(t,r),\\
        a(t,r) &\lesssim t^{-1+p/(p-1)\beta}\,r^{(p-2)/(p-1)},\\
        a(t,r)\,U(t,r) &\asymp  t^{-1-d/\beta+p/(p-1)\beta}\,r^{(p-2)/(p-1)}\,S(t,r)^{\gamma+1},\\
        a(t,r)\,U(t,r) &\lesssim t^{-1-d/\beta+p/(p-1)\beta}\,r^{(p-2)/(p-1)}.
    \end{align}
    \item\label{La:Zwei} For every $\lambda\in (0,1)$ there exists a  constant $c>0$ such that for every $t>0$ and every $r\in (0,\lambda R(t))$,
    \begin{align}
        c\, t^{-1+p/(p-1)\beta}\,r^{(p-2)/(p-1)}&\leq a(t,r),\\
        c^{\gamma+1}\,t^{-1-d/\beta+p/(p-1)\beta}\,r^{(p-2)/(p-1)}&\leq a(t,r)\,U(t,r).
    \end{align}
\end{enumerate}
\end{lemma}
\begin{proof} Given $t>0$ and $r\in (0,R(t))$, by \eqref{Eq:U} we have
\begin{align}\label{Eq:US}
U(t,r)=t^{-d/\beta}\,S(t,r)^\gamma
\end{align}
and $S(t,r)$ is positive. Differentiating this identity in the second argument yields
\begin{align}
    U_r(t,r) = t^{-d/\beta}\,\gamma\,S(t,r)^{\gamma-1}\,S_r(t,r) = -\gamma\kappa\frac{ p}{p-1}\,t^{-d/\beta-p/(p-1)\beta}\,r^{1/(p-1)}\,S(t,r)^{\gamma-1}.
\end{align}
Taking absolute values and raising the result to the power $p-2$,
\begin{align}
    \big\vert U_r(t,r)\big\vert^{p-2} = \Big[\gamma\kappa\frac{p}{p-1}\Big]^{p-2}\,t^{-(p-2)(d/\beta+p/(p-1)\beta)}\,r^{(p-2)/(p-1)}\,S(t,r)^{(p-2)(\gamma-1)}.
\end{align}
On the other hand, by \eqref{Eq:U} we obtain
\begin{align}
    U(t,r)^{(m-1)(p-1)} = t^{-(m-1)(p-1)d/\beta}\,S(t,r)^{(m-1)(p-1)\gamma}.
\end{align}
This yields
\begin{align}
    a(t,r) &= m^{p-1}\,\Big[\gamma\kappa\frac{p}{p-1}\Big]^{p-2}\,t^{-(p-2)(d/\beta+p/(p-1)\beta)}\,t^{-(m-1)(p-1)d/\beta}\\
    &\qquad\qquad \times r^{(p-2)/(p-1)}\,S(t,r)^{(p-2)(\gamma-1)}\,S(t,r)^{(m-1)(p-1)\gamma}\\
    &=m^{p-1}\,\Big[\gamma\kappa\frac{p}{p-1}\Big]^{p-2}\,t^{-1+p/(p-1)\beta}\,r^{(p-2)/(p-1)}\,S(t,r).
\end{align}
This immediately implies the first and third claim from \ref{La:Eins}. The other two follow from the simple observation $S\leq C$ on $(0,\infty)\times \R_+$, recalling $\kappa$ is positive.
To show item \ref{La:Zwei}, by \ref{La:Eins} it suffices to find a suitable lower bound on $S(t,r)$ for every $t>0$ and every $r\in(0,\lambda R(t))$ as given. However, this is straightforward since, when $t$ and $r$ belong to these ranges,
\begin{align}
    S(t,r) &=C-\kappa\big[t^{-1/\beta}\,r\big]^{p/(p-1)} \geq C-\kappa\big[t^{-1/\beta}\,\lambda R(t)\big]^{p/(p-1)}= C\,\big[1-\lambda^{p/(p-1)}\big].
\end{align}
This completes the proof.
\end{proof}

\section{Radial process and Tanaka--Meyer semimartingale decomposition}\label{Sec:RadialTanakaMeyer}
Now we show the radial parts of the processes  in question constitute semi\-martingales by proving explicit Tanaka--Meyer formulas. These involve the induced generator \eqref{Eq:Generator}, notably the coefficient $a$ from \eqref{Eq:a}; we recall moreover the definition \eqref{Eq:sgn} of the vector-valued sign function $\sgn$. We also characterize  non\-triviality of the associated local times. For a comparison of our results to Euclidean Brownian motion, we refer to \cref{Re:BMLocal}. 

Given a starting point $\smash{x\in\R^d}$, let $X$ be a probabilistically weak solution to \eqref{Eq:ClosedMcKeanVlasov}. It will always be modeled on a fixed stochastic basis $\smash{(\Omega,\scrF,\scrF_\bullet,\PPP^x)}$ on which it solves \eqref{Eq:ClosedMcKeanVlasov} and its driving process $W$ is a Brownian motion.
\subsection{Itô formula for approximate radial process} In the sequel, fix a ``center'' $\smash{o\in\R^d}$. Define the function $\rO\colon\R^d\to\R_+$ by
\begin{align}\label{Eq:Distance}
    \rO(y) := \vert y-o\vert.
\end{align}
Given $\varepsilon>0$, we define the approximate distance function  $\rO_\varepsilon\colon\R^d\to\R_+$ by
\begin{align}
    \rO_\varepsilon(y) := \sqrt{\vert y-o\vert^2+\varepsilon^2}.
\end{align}
A direct computation reveals everywhere on $\smash{\R^d}$,
\begin{align}\label{Eq:Derreps}
\begin{split}
    \nabla \rO_\varepsilon &= \frac{\cdot -o}{\rO_\varepsilon},\\
    \Delta \rO_\varepsilon &= \frac{d-1}{\rO_\varepsilon} +\frac{\varepsilon^2}{{\rO_\varepsilon}^3},
\end{split}
\end{align}
identities which continue to hold on $\smash{\R^d\setminus\{o\}}$ when $\varepsilon$ is zero.
\begin{lemma}[Itô formula for approximate radial process]\label{Le:Itoeps1} For every $\varepsilon>0$ and every point $\smash{x\in\R^d}$, the approximate radial process $\rO_\varepsilon\circ X$ obeys
\begin{align}\label{Eq:repsIto}
\begin{split}
    \rmd \rO_\varepsilon(X_t) &= \sqrt{2a(t,\vert X_t-x\vert)}\,\frac{X_t-o}{\rO_\varepsilon(X_t)}\cdot\rmd W_t\\
    &\qquad\qquad + a(t,\vert X_t-x\vert)\,\Big[\frac{d-1}{\rO_\varepsilon(X_t)}+\frac{\varepsilon^2}{\rO_\varepsilon(X_t)^3}\Big]\d t\\
    &\qquad\qquad + a_r(t,\vert X_t-x\vert)\sgn(X_t-x)\cdot\frac{X_t-o}{\rO_\varepsilon(X_t)}\d t\quad\PPP^{x}\textnormal{\textit{-a.s.}}
    \end{split}
\end{align}
\end{lemma}
\begin{proof} Using \eqref{Eq:Integrability}, it is not hard to show that all Itô integrals appearing above are true martingales and that all Lebesgue integrals do  exist in $\R$.
Thanks to  \eqref{Eq:ClosedMcKeanVlasov}, the quadratic variation of $X$ is the process $\smash{[X,X]}$ of  $d\times d$-matrices whose entries solve
\begin{align}
    \rmd [X^i,X^j]_t &= 2a(t,\vert X_t-x\vert)\d[W^i,W^j]_t =2a(t,\vert X_t-x\vert)\,\delta_{ij}\d t\quad \PPP^{x}\textnormal{-a.s.}
\end{align}
Since $\rO_\varepsilon$ is smooth, Itô's formula combines  with \eqref{Eq:ClosedMcKeanVlasov} to entail
\begin{align}
    \rmd \rO_\varepsilon(X_t) &= \nabla \rO_\varepsilon(X_t)\cdot \rmd X_t + \frac{1}{2}\nabla^2\rO_\varepsilon(X_t):\rmd[X,X]_t\\
    &=\nabla \rO_\varepsilon(X_t)\cdot \rmd X_t + a(t,\vert X_t-x\vert)\,\Delta \rO_\varepsilon(X_t)\d t\\
    &= \sqrt{2a(t,\vert X_t-x\vert)}\,\nabla \rO_\varepsilon(X_t)\cdot \rmd W_t +  a_r(t,\vert X_t-x\vert)\sgn(X_t-x)\cdot\nabla \rO_\varepsilon(X_t)\d t\\
    &\qquad\qquad + a(t,\vert X_t-x\vert)\,\Delta \rO_\varepsilon(X_t)\d t \quad\PPP^{x}\textnormal{-a.s.}
\end{align}
Inserting the formulas from \eqref{Eq:Derreps} yields the claim.
\end{proof}

\begin{corollary}[Itô formula for centered approximate radial process]\label{Cor:Itoeps} For every $\varepsilon>0$, the centered approximate radial process $\rO_\varepsilon\circ X$ satisfies
\begin{align}
    \rmd \rO_\varepsilon(X_t) &= \sqrt{2a(t, \rO(X_t))}\,\frac{X_t-o}{\rO_\varepsilon(X_t)}\cdot\rmd W_t + a(t,\rO(X_t))\,\Big[\frac{d-1}{\rO_\varepsilon(X_t)}+\frac{\varepsilon^2}{\rO_\varepsilon(X_t)^3}\Big]\d t\\
    &\qquad\qquad + a_r(t,\rO(X_t)) \,\frac{\rO(X_t)}{\rO_\varepsilon(X_t)}\d t\quad\PPP^{o}\textnormal{\textit{-a.s.}}
\end{align}
\end{corollary}
In the sequel, we will carefully send $\smash{\varepsilon\to 0^+}$ in the above Itô formulas. The resulting SDE will constitute the desired Tanaka--Meyer semi\-martingale decomposition of the radial process $\smash{\rO\circ X}$. We also characterize the appearing local time precisely. We will have to deal with three mutually exclusive cases, namely
\begin{itemize}
    \item the centered case when the ``center'' $o$ is the starting point of $X$,
    \item the uncentered case in dimensions at least two, and
    \item the uncentered case in dimension one.
\end{itemize}
\subsection{Centered case, arbitrary dimension} We briefly recall a standard notion of convergence of processes.
\begin{definition}[ucp-Convergence]\label{Def:ucp} Let $x\in\R^d$ be given. Let $(Y^\varepsilon)_{\varepsilon>0}$ constitute a net of $\scrF_\bullet$-semi\-martingales. Let $Y$ be a further $\scrF_\bullet$-semimartingale. We say $(Y^\varepsilon)_{\varepsilon>0}$ converges to $Y$ uniformly on compact sets in probability, briefly \emph{ucp} and symbolically
\begin{align}
    Y = \PPP^x\textnormal{-}\textnormal{ucp-}\!\!\!\lim_{\varepsilon\to 0^+}Y^\varepsilon,
\end{align}
if for every $b\in\R_+$, the following convergence holds in $\PPP^x$-probability:
\begin{align}
    \lim_{\varepsilon\to 0^+}\sup\{\big\vert Y_t^\varepsilon-Y_t\big\vert : t\in [0,b]\} =0.
\end{align}
\end{definition}
Our results  from \cref{Th:CenteredRadial,Th:RadialTwo} below will follow by taking the ucp-limits as $\varepsilon\to 0^+$ of \cref{Le:Itoeps1,Cor:Itoeps}. As indicated above, the conclusions and proofs differ depending  on whether the process in question starts at the ``center'' $o$ or somewhere else. For the terms  where this distinction will not matter, we save the following  result.
\begin{proposition}[Three ucp-convergences]\label{Pr:Threeucp} For every $x\in\R^d$,  there is a standard $\smash{\scrF_\bullet}$-Brownian motion $B$ on $\R$ such that
\begin{align}
    &\PPP^{x}\textnormal{-}\textnormal{ucp-}\!\!\!\lim_{\varepsilon\to 0^+} \Big[\rO_\varepsilon \circ X -\rO_\varepsilon(x)-\int_{[0,\bullet]} \sqrt{2a(s,\vert X_s-x\vert)}\,\frac{X_s-o}{\rO_\varepsilon(X_s)}\cdot\rmd W_s\\
    &\qquad\qquad\qquad\qquad -\int_{[0,\bullet]} a_r(s,\vert X_s-x\vert)\,\sgn(X_s-x)\cdot\frac{X_s-o}{\rO_\varepsilon(X_s)}\d s\Big]\\
    &\qquad\qquad = \rO\circ X -\rO(x)- \int_{[0,\bullet]}\sqrt{2a(s,\vert X_s-x\vert)} \d B_s\\
    &\qquad\qquad\qquad\qquad - \int_{[0,\bullet]} a_r(s,\vert X_s-x\vert)\,\sgn(X_s-x)\cdot\sgn(X_s-o)\d s.
\end{align}
\end{proposition}
\begin{proof} We will show ucp-convergence of the three summands in question separately. In several steps, we will tacitly use that all local martingales considered below are in fact true martingales, which is a simple consequence of \eqref{Eq:Integrability} and Itô's isometry (applied to the Itô integrals with respect to the driving Brownian motion $W$).

First, since $\vert \rO_\varepsilon-\rO\vert\leq \varepsilon$ everywhere on $\R^d$, we clearly have
\begin{align}
    \PPP^{x}\textnormal{-}\textnormal{ucp-}\!\!\!\lim_{\varepsilon\to 0^+} \rO_\varepsilon \circ X-\rO_\varepsilon (x) =\rO\circ X-\rO(x).
\end{align}

Second, we claim
\begin{align}
&\PPP^{x}\textnormal{-}\textnormal{ucp-}\!\!\!\lim_{\varepsilon\to 0^+}\int_{[0,\bullet]} \sqrt{2a(s,\vert X_s-x\vert)}\,\frac{X_s-o}{\rO_\varepsilon(X_s)}\cdot\rmd W_s\\
    &\qquad\qquad = \int_{[0,\bullet]}\sqrt{2a(s,\vert X_s-x\vert)} \sgn(X_s-o)\cdot\rmd W_s.
\end{align}
First note the process on the right-hand side is $\smash{\PPP^{x}}$-a.s.~well-defined. This follows  from \eqref{Eq:Integrability}, Itô's isometry, and the following occupation time identity based on  Fubini's theorem that implies the event $\{\rO\circ X=0\}$ is $(\Leb^1\mres\R_+)\otimes\PPP^{x}$-negligible, where $t\in\R_+$:
\begin{align}\label{Eq:Occup}
    \EEE^{x}\Big[\!\int_{[0,t]} 1_{\{\rO(X_s)=0\}}\d s\Big] = \int_{[0,t]}\PPP^{x}[X_s=o]\d s=0;
\end{align}
in the last identity, we used the law of $X_s$ is $\smash{\Leb^d}$-absolutely continuous for every $s>0$ by \eqref{Eq:ClosedMcKeanVlasov}.
Given $b\in\R_+$ and $\delta >0$, combining Doob's $L^2$-inequality, cf.~e.g.~Revuz--Yor \cite{revuz-yor1999}*{Thm. II.1.7}, with Itô's isometry gives
\begin{align}
    &\PPP^{x}\Big[\!\sup\!\Big\lbrace \Big\vert\!\int_{[0,t]} \sqrt{2a(s,\vert X_s-x\vert)}\,\Big[\frac{X_s-o}{\rO_\varepsilon(X_s)}-\sgn(X_s-o)\Big]\cdot\rmd W_s\Big\vert : t\in[0,b]\Big\rbrace \geq \delta \Big]\\
    &\qquad\qquad \leq \frac{2}{\delta^2}\,\EEE^{x}\Big[\Big[\!\int_{[0,b]} \sqrt{a(s,\vert X_s-x\vert)}\,\Big[\frac{X_s-o}{\rO_\varepsilon(X_s)} - \sgn(X_s-o)\Big]\cdot\rmd W_s\Big]^2\Big]\\
    &\qquad\qquad = \frac{2}{\delta^2}\,\EEE^{x}\Big[\!\int_{[0,b]} a(s,\vert X_s-x\vert)\,\Big\vert  \frac{X_s-o}{\rO_\varepsilon(X_s)} - \sgn(X_s-o)\Big\vert^2\d s\Big].
\end{align}
Using \eqref{Eq:Integrability} and Lebesgue's dominated convergence theorem yields
\begin{align}
    \lim_{\varepsilon\to 0^+}\EEE^{x}\Big[\!\int_{[0,b]} a(s,\vert X_s-x\vert)\,\Big\vert  \frac{X_s-o}{\rO_\varepsilon(X_s)} - \sgn(X_s-o)\Big\vert^2\d s\Big]=0,
\end{align}
to establish the second convergence. To conclude, we claim the process $B$ defined by
\begin{align}\label{Eq:Bdef}
    B_t:= \int_{[0,t]}\sgn(X_s-o)\cdot\rmd W_s
\end{align}
constitutes a standard $\smash{\scrF_\bullet}$-Brownian motion under $\smash{\PPP^x}$.  This, however, is clear from Lévy's characterization of Brownian motion, cf.~e.g.~Revuz--Yor \cite{revuz-yor1999}*{Thm.~IV.3.6}, noting  $B_0=0$ and, for every $t\in\R_+$, \eqref{Eq:Occup} implies
\begin{align}
    [B]_t = \int_{[0,t]}\big\vert\! \sgn(X_s-o)\big\vert^2\d s = t\quad\PPP^{x}\textnormal{-a.s.}
\end{align}

The third claimed convergence
\begin{align}
    &\PPP^{x}\textnormal{-}\textnormal{ucp-}\!\!\!\lim_{\varepsilon\to 0^+}\int_{[0,\bullet]} a_r(s,\vert X_s-x\vert)\sgn(X_s-x)\cdot\frac{X_s-o}{\rO_\varepsilon(X_s)}\d s\\
    &\qquad\qquad = \int_{[0,\bullet]} a_r(s,\vert X_s-x\vert)\sgn(X_s-x)\cdot\sgn(X_s-o)\d s
\end{align}
follows from Markov's inequality and Lebesgue's dominated convergence theorem, arguing as in the previous step and using \eqref{Eq:Integrability}.
\end{proof}
The remaining drift terms from \cref{Le:Itoeps1,Cor:Itoeps} involving reciprocals of the radial process require the indicated case distinction for the starting point. When the latter coincides with $o$, no further hypotheses are needed to ensure that one can ``set $\varepsilon$ to zero'' in the above Itô formulas. The reader may find it useful to consult the much simpler proof of \cref{Th:RadialTwo} first to get a feeling which additional challenges we face in the proof of \cref{Th:CenteredRadial}, where we cannot use ``crude'' estimates, cf.~\eqref{Eq:ETADEF}.

\begin{theorem}[Tanaka--Meyer formula for centered radial process]\label{Th:CenteredRadial} There is a standard $\scrF_\bullet$-Brownian motion $B$ on $\R$ such that
\begin{align}
    \rmd \rO(X_t)&=\sqrt{2a(t,\rO(X_t))}\d B_t + a(t,\rO(X_t))\,\frac{d-1}{\rO(X_t)}\d t + a_r(t,\rO(X_t))\d t\\
    &= \sqrt{2a(t,\rO(X_t))}\d B_t + \sfL_{p,m}^o\rO(X_t)\d t\quad\PPP^{o}\textnormal{\textit{-a.s.}},
\end{align}
where $\smash{\sfL_{p,m}^o}$ is from \eqref{Eq:Generator}.
\end{theorem}
\begin{proof}  By \cref{Cor:Itoeps,Pr:Threeucp}, it suffices to prove
\begin{align}
    \PPP^{o}\textnormal{-}\textnormal{ucp-}\!\!\!\lim_{\varepsilon\to 0^+}\int_{[0,\bullet]} a(s,\rO(X_s))\,\Big[\frac{d-1}{\rO_\varepsilon(X_s)}+\frac{\varepsilon^2}{\rO_\varepsilon(X_s)^3}\Big]\d s = \int_{[0,\bullet]} a(s,\rO(X_s))\,\frac{d-1}{\rO(X_s)} \d s
\end{align}
and that the Lebesgue integral on the right-hand side is finite $\smash{\PPP^{o}}$-a.s. Let $b\in\R_+$ and $\delta>0$. Using \eqref{Eq:Occup} and Markov's inequality,
\begin{align}
    &\PPP^{o}\Big[\!\sup\!\Big\lbrace\Big\vert\!\int_{[0,t]} a(s, \rO(X_s))\,\Big[\frac{d-1}{\rO_\varepsilon(X_s)}+\frac{\varepsilon^2}{\rO_\varepsilon(X_s)^3} - \frac{d-1}{\rO(X_s)}\Big]\d s\Big\vert : t\in[0,b]\Big\rbrace \geq 2\delta\Big]\\
    &\qquad\qquad \leq \PPP^{o}\Big[\!\int_{[0,b]}  a(s, \rO(X_s))\, \Big\vert\frac{d-1}{\rO_\varepsilon(X_s)} - \frac{d-1}{\rO(X_s)}\Big\vert\d s \geq \delta\Big]\\
    &\qquad\qquad\qquad\qquad +\PPP^{o}\Big[\!\int_{[0,b]}a(s, \rO(X_s))\,\frac{\varepsilon^2}{\rO_\varepsilon(X_s)^3}\d s\geq \delta\Big]\\
    &\qquad\qquad \leq\frac{d-1}{\delta}\underbrace{\EEE^{o}\Big[\!\int_{[0,b]}  a(s, \rO(X_s))\, \Big\vert\frac{1}{\rO_\varepsilon(X_s)} - \frac{1}{\rO(X_s)}\Big\vert\d s\Big]}_{\rmA_\varepsilon}\\
    &\qquad\qquad\qquad\qquad + \frac{1}{\delta}\underbrace{\EEE^{o}\Big[\!\int_{[0,b]}a(s, \rO(X_s))\,\frac{\varepsilon^2}{\rO_\varepsilon(X_s)^3}\d s\Big]}_{\rmB_\varepsilon}\!.
\end{align}

\textit{Claim} A. We claim
\begin{align}\label{Eq:Aeconv}
    \lim_{\varepsilon\to 0^+}\rmA_\varepsilon=0.
\end{align}
Indeed, by our hypothesis for the starting point, we have $\smash{\Law\,X_s=U(s,\rO) \,\Leb^d}$ for every $s\in[0,b]$ by \eqref{Eq:ClosedMcKeanVlasov}; thus, employing Fubini's theorem and using polar coordinates,
\begin{align}\label{Eq:SecondExp}
\begin{split}
\rmA_\varepsilon &= \EEE^{o}\Big[\!\int_{[0,b]} a(s,\rO(X_s))\,\Big[\frac{1}{\rO(X_s)}-\frac{1}{\rO_\varepsilon(X_s)}\Big]\d s\Big]\\
    &\leq \EEE^{o}\Big[\!\int_{[0,b]} a(s,\rO(X_s))\,\frac{1}{\rO(X_s)}\d s\Big]\\
    &=\int_{[0,b]}\int_{B_{R(s)}(o)} a(s,\rO(y))\,\frac{1}{\rO(y)}\,U(s,\rO(y))\d y\d s\\
    & = \omega_{d-1}\int_{[0,b]}\int_{[0, R(s)]} a(s,r)\,r^{d-2}\,U(s,r)\d r\d s.
\end{split}
\end{align}
By \cref{Le:aU-bounds-m}, we obtain
\begin{align}\label{Eq:Ints}
\begin{split}
    &\int_{[0,b]}\int_{[0, R(s)]} a(s,r)\,r^{d-2}\,U(s,r)\d r\d s\\
    &\qquad\qquad \lesssim \int_{[0,b]} s^{-1-d/\beta+ p/(p-1)\beta}\int_{[0, R(s)]} r^{d-2+(p-2)/(p-1)}\d r\d s.
    \end{split}
\end{align}
Our standing hypothesis \eqref{Eq:ranges} implies
\begin{align}\label{Eq:Expd2}
    d-2+\frac{p-2}{p-1}>-1.
\end{align}
The latter inner integral from \eqref{Eq:Ints} is  thus finite; by definition \eqref{Eq:Rdef} of $R$, we find  the right-hand side of \eqref{Eq:Ints} is no larger than
\begin{align}\label{Eq:Larger!}
\begin{split}
    \int_{[0,b]}s^{-1-d/\beta+p/(p-1)\beta}\,R(s)^{d-1+(p-2)/(p-1)}\d s \asymp \int_{[0,b]} s^{-1+1/\beta}\d s,
\end{split}
\end{align}
which is clearly finite. In particular, by \eqref{Eq:SecondExp} the integrand of $\rmA_\varepsilon$ admits a $\smash{\PPP^{o}}$-integrable dominant function. By Lebesgue's dominated convergence theorem, we thus get \eqref{Eq:Aeconv}.

\textit{Claim} B. We claim
\begin{align}\label{Eq:Beconv}
    \lim_{\varepsilon\to 0^+}\rmB_\varepsilon=0.
\end{align}
Arguing as for \eqref{Eq:SecondExp},
\begin{align}
    \rmB_\varepsilon &= \EEE^{o}\Big[\!\int_{[0,b]}a(s,\rO(X_s))\,\frac{\varepsilon^2}{\rO_\varepsilon(X_s)^3}\d s\Big]\\
    &=\int_{[0,b]}\int_{B_{R(s)}(o)} a(s,\rO(y))\,\frac{\varepsilon^2}{\sqrt{\rO(y)^2+\varepsilon^2}^3}\,U(s,\rO(y))\d y\d s\\
    &=\omega_{d-1}\int_{[0,b]}\int_{[0,R(s)]} a(s,r)\,\frac{\varepsilon^2\,r^{d-1}}{\sqrt{r^2+\varepsilon^2}^3}\,U(s,r)\d r\d s.
\end{align}
Let $R^{-1}$ be  from \eqref{Eq:R-1def}. When $\varepsilon$ is sufficiently small we have $R^{-1}(\varepsilon) < b$ and thus
\begin{align}
    &\int_{[0,b]}\int_{[0,R(s)]}a(s,r)\,\frac{\varepsilon^2\,r^{d-1}}{\sqrt{r^2+\varepsilon^2}^3}\,U(s,r)\d r\d s\\
    &\qquad\qquad = \underbrace{\int_{[0,R^{-1}(\varepsilon)]}\int_{[0,R(s)]}a(s,r)\,\frac{\varepsilon^2\,r^{d-1}}{\sqrt{r^2+\varepsilon^2}^3}\,U(s,r)\d r\d s}_{\rmI_\varepsilon}\\
    &\qquad\qquad\qquad\qquad + \underbrace{\int_{[R^{-1}(\varepsilon),b]}\int_{[0,R(s)]}a(s,r)\,\frac{\varepsilon^2\,r^{d-1}}{\sqrt{r^2+\varepsilon^2}^3}\,U(s,r)\d r\d s}_{\rmI\rmI_\varepsilon}\!.
\end{align}

\textit{Part} B.I. To estimate the first integral, we simply use
\begin{align}\label{Eq:eps3}
    \sqrt{r^2+\varepsilon^2}^3 \geq \varepsilon^3
\end{align}
for every $r\in\R$, which yields
\begin{align}
    \rmI_\varepsilon \leq \frac{1}{\varepsilon}\int_{[0,R^{-1}(\varepsilon)]}\int_{[0,R(s)]}a(s,r)\,r^{d-1}\,U(s,r)\d r\d s.
\end{align}
Arguing as for the bound of \eqref{Eq:SecondExp} by \eqref{Eq:Ints},
\begin{align}
    \rmI_\varepsilon\lesssim \frac{1}{\varepsilon}\int_{[0,R^{-1}(\varepsilon)]} s^{-1+2/\beta}\d s \asymp  \frac{R^{-1}(\varepsilon)^{2/\beta}}{\varepsilon}\asymp\varepsilon;
\end{align}
we thus deduce
\begin{align}
    \lim_{\varepsilon\to 0^+}\rmI_\varepsilon=0.
\end{align}

\textit{Part} B.II. To estimate the other integral, we perform one more decomposition to get
\begin{align}
    \rmI\rmI_\varepsilon &= \underbrace{\int_{[R^{-1}(\varepsilon),b]}\int_{[0,\varepsilon]}a(s,r)\,\frac{\varepsilon^2\,r^{d-1}}{\sqrt{r^2+\varepsilon^2}^3}\,U(s,r)\d r\d s}_{\rma_\varepsilon}\\
    &\qquad\qquad + \underbrace{\int_{[R^{-1}(\varepsilon),b]}\int_{[\varepsilon,R(s)]}a(s,r)\,\frac{\varepsilon^2\,r^{d-1}}{\sqrt{r^2+\varepsilon^2}^3}\,U(s,r)\d r\d s}_{\rmb_\varepsilon}\!.
\end{align}

\textit{Subpart} B.II.a. Using \eqref{Eq:eps3}, \cref{Le:aU-bounds-m}, and finally \eqref{Eq:Expd2},
\begin{align}
    \rma_\varepsilon &\leq \frac{1}{\varepsilon}\int_{[R^{-1}(\varepsilon),b]}\int_{[0,\varepsilon]}a(s,r)\,r^{d-1}\,U(s,r)\d r\d s\\
    &\lesssim \int_{[R^{-1}(\varepsilon),b]}s^{-1-d/\beta +p/(p-1)\beta}\,\varepsilon^{d+(p-2)/(p-1)-1}\d s.
\end{align}
Again by \eqref{Eq:Expd2}, the exponent of $\varepsilon$ is positive. In particular, we may and will fix $\delta\in (0,1)$ sufficiently small with the property
\begin{align}
    d+\frac{p-2}{p-1}-1-\delta>0.
\end{align}
Combining this with the inequality $\varepsilon\leq R(s)$ for every $s\in [R^{-1}(\varepsilon), b]$,
\begin{align}
    \rma_\varepsilon &\leq \varepsilon^\delta\int_{[R^{-1}(\varepsilon),b]} s^{-1-d/\beta+p/(p-1)\beta}\,\varepsilon^{d+(p-2)/(p-1)-1-\delta}\d s\\
    &\leq\varepsilon^\delta\int_{[R^{-1}(\varepsilon),b]} s^{-1-d/\beta+p/(p-1)\beta}\,R(s)^{d+(p-2)/(p-1)-1-\delta}\d s\\
    &\lesssim  \varepsilon^\delta\int_{[0,b]} s^{-1+(1-\delta)/\beta}\d s.
\end{align}
Since $\delta\in (0,1)$, the latter integral is finite and
\begin{align}
    \lim_{\varepsilon\to 0^+}\rma_\varepsilon = 0.
\end{align}

\textit{Subpart} B.II.b. On the other hand, using the inequality
\begin{align}
    \sqrt{r^2+\varepsilon^2}^3\geq r^3
\end{align}
for every $r\in\R_+$ and \cref{Le:aU-bounds-m},
\begin{align}\label{Eq:starting}
\begin{split}
    \rmb_\varepsilon&\leq \varepsilon^2\int_{[R^{-1}(\varepsilon),b]}\int_{[\varepsilon,R(s)]}a(s,r) \,r^{d-4}\,U(s,r)\d r\d s\\
    &\lesssim \varepsilon^2\int_{[R^{-1}(\varepsilon),b]} s^{-1-d/\beta+p/(p-1)\beta} \int_{[\varepsilon,R(s)]} r^\alpha\d r\d s,
    \end{split}
\end{align}
where we abbreviate
\begin{align}\label{Eq:alpha}
    \alpha := d-4 +\frac{p-2}{p-1} >-3;
\end{align}
the last inequality comes from \eqref{Eq:Expd2}.
We now distinguish two cases depending on $\alpha$.

\textit{Subsubpart} B.II.b.1. If $\alpha$ is distinct from  $-1$,
\begin{align}
    \rmb_\varepsilon &\lesssim \varepsilon^2\int_{[R^{-1}(\varepsilon),b]}s^{-1-d/\beta+p/(p-1)\beta}\,R(s)^{\alpha+1}\d s\\
    &\qquad\qquad + \varepsilon^{\alpha+3}\int_{[R^{-1}(\varepsilon),b]} s^{-1-d/\beta+p/(p-1)\beta}\d s.
\end{align}
On the one hand, using the definition \eqref{Eq:alpha} of $\alpha$,
\begin{align}
    &\varepsilon^2\int_{[R^{-1}(\varepsilon),b]} s^{-1-d/\beta+p/(p-1)\beta}\,R(s)^{\alpha+1}\d s\asymp \varepsilon^2\int_{[R^{-1}(\varepsilon),b]} s^{-1-1/\beta}\d s \lesssim \frac{\varepsilon^2}{R^{-1}(\varepsilon)^{1/\beta}}\asymp\varepsilon,
\end{align}
which converges to zero as $\varepsilon\to 0^+$. On the other hand,
\begin{align}
    &\varepsilon^{\alpha+3}\int_{[R^{-1}(\varepsilon),b]} s^{-1-d/\beta+p/(p-1)\beta}\d s\\
    &\qquad\qquad\lesssim_b \begin{cases}\varepsilon^{\alpha+3}+\varepsilon^{\alpha+3}\,R^{-1}(\varepsilon)^{-d/\beta+p/(p-1)\beta} & \textnormal{if }d\neq p/(p-1),\\
    \varepsilon^{\alpha+3}+\varepsilon^{\alpha+3}\,\big\vert\!\log R^{-1}(\varepsilon)\big\vert & \textnormal{otherwise}
    \end{cases}
\end{align}
The first summands on the right-hand side clearly converge to zero since $\alpha +3>0$. In addition, using the definition \eqref{Eq:alpha} of $\alpha$ again,
\begin{align}
    \varepsilon^{\alpha+3}\,R^{-1}(\varepsilon)^{-d/\beta+p/(p-1)\beta} \asymp \varepsilon
\end{align}
provided $d\neq p/(p-1)$ and
\begin{align*}
    \varepsilon^{\alpha+3}\,\big\vert\!\log R^{-1}(\varepsilon)\big\vert \asymp_{d,p,m}\varepsilon^{\alpha+3}\,\big\vert\!\log\varepsilon\big\vert
\end{align*}
otherwise, and consequently
\begin{align}\label{Eq:II'convergence}
\lim_{\varepsilon\to 0^+}\rmb_\varepsilon =0.
\end{align}

\textit{Subsubpart} B.II.b.2. If $\alpha$ equals $-1$,  algebraic manipulations based on \eqref{Eq:alpha} show
\begin{align}
    \frac{p}{p-1}=d-1.
\end{align}
Therefore, starting again from \eqref{Eq:starting},
\begin{align}
    \rmb_\varepsilon &\lesssim \varepsilon^2\int_{[R^{-1}(\varepsilon),b]}s^{-1-d/\beta+(d-1)/\beta}\, \big\vert \!\log R(s)\big\vert\d s\\
    &\qquad\qquad + \varepsilon^2\int_{[R^{-1}(\varepsilon),b]}s^{-1-d/\beta+(d-1)/\beta} \,\big\vert\!\log \varepsilon\big\vert\d s\\
    &\lesssim \varepsilon^2\,\big\vert\!\log \varepsilon \big\vert\int_{[R^{-1}(\varepsilon),b]} s^{-1-1/\beta}\d s
\end{align}
for every sufficiently small $\varepsilon$; in the second inequality, we used that since $R^{-1}(\varepsilon)\to 0^+$ as $\varepsilon\to 0^+$ and  as $\log\circ R$ diverges to $-\infty$ near zero yet is bounded on $[1,b]$ if this interval is nonempty, for every sufficiently small $\varepsilon$ and every $s\in [R^{-1}(\varepsilon),b]$ we have
\begin{align}
    \big\vert\! \log R(s)\big\vert \leq \big\vert\!\log R(R^{-1}(\varepsilon))\big\vert = \big\vert \!\log\varepsilon\big\vert.
\end{align}
Using the definition \eqref{Eq:R-1def} of $R^{-1}$,
\begin{align}
    &\varepsilon^2\,\big\vert\!\log\varepsilon\big\vert\int_{[R^{-1}(\varepsilon),b]}s^{-1-1/\beta}\d s \lesssim \varepsilon^2\,\big\vert\!\log\varepsilon\big\vert \,R^{-1}(\varepsilon)^{-1/\beta} \asymp\varepsilon\,\big\vert\!\log\varepsilon\big\vert,
\end{align}
to terminate the proof of \eqref{Eq:II'convergence}.

In summary, the accomplishments of Parts B.I and B.II prove \eqref{Eq:Beconv}. In turn, Claims A and B establish the desired first identity.

The second identity follows by a straightforward  computation, justified $\PPP^{o}$-a.s.~by the preceding arguments.
\end{proof}

\begin{remark}[Explicit driving Brownian motion] If the  Brownian motion $W$ driving \eqref{Eq:ClosedMcKeanVlasov}  can be chosen independently of the starting point, the identity \eqref{Eq:Bdef} shows that in \cref{Th:CenteredRadial,Th:RadialTwo,Th:RadialOne}, the one-dimensional standard Brownian motion driving the radial process  can be chosen independently of the starting point.
\end{remark}

In particular, the centered radial process has trivial local time in \emph{any} dimension, which is unlike Euclidean Brownian motion; cf.~\cref{Re:BMLocal} below for details. If the process in question starts outside $o$, in accordance with Euclidean Brownian motion one still obtains a trivial local time as long as the dimension $d$ is at least two. The one-dimensional case is different and addressed in \cref{Th:RadialOne}.

\subsection{Uncentered case, dimension at least two}

\begin{theorem}[Tanaka--Meyer formula for high-dimensional radial process]\label{Th:RadialTwo} Let    $d\geq 2$. Then for every $x\in\R^d\setminus\{o\}$ there is a standard $\smash{\scrF_\bullet}$-Brownian motion $B$ on $\R$ such that
\begin{align}
    \rmd \rO(X_t) &= \sqrt{2a(t,\vert X_t-x\vert)}\d B_t + a(t,\vert X_t-x\vert)\,\frac{d-1}{\rO(X_t)}\d t\\
    &\qquad\qquad + a_r(t,\vert X_t-x\vert)\sgn(X_t-x)\cdot \sgn(X_t-o)\d t\\
    &=\sqrt{2a(t,\vert X_t-x\vert)}\d B_t +\sfL_{p,m}^x\rO(X_t)\d t\quad\PPP^{x}\textnormal{\textit{-a.s.}},
\end{align}
where $\smash{\sfL_{p,m}^x}$ is from \eqref{Eq:Generator}.
\end{theorem}
\begin{proof}  By \cref{Cor:Itoeps,Pr:Threeucp}, it suffices to prove
\begin{align}
    &\PPP^{x}\textnormal{-}\textnormal{ucp-}\!\!\!\lim_{\varepsilon\to 0^+}\int_{[0,\bullet]} a(s,\vert X_s-x\vert)\,\Big[\frac{d-1}{\rO_\varepsilon(X_s)}+\frac{\varepsilon^2}{\rO_\varepsilon(X_s)^3}\Big]\d s\\
    &\qquad\qquad= \int_{[0,\bullet]} a(s,\vert X_s-x\vert)\,\frac{d-1}{\rO(X_s)} \d s
\end{align}
and that the Lebesgue integral on the right-hand side is finite $\smash{\PPP^{x}}$-a.s. Let $b\in\R_+$ and $\delta>0$. Arguing as in the proof of \cref{Th:CenteredRadial},
\begin{align}
    &\PPP^{x}\Big[\!\sup\!\Big\lbrace\Big\vert\!\int_{[0,t]} a(s,\vert X_s-x\vert)\,\Big[\frac{d-1}{\rO_\varepsilon(X_s)}+\frac{\varepsilon^2}{\rO_\varepsilon(X_s)^3} - \frac{d-1}{\rO(X_s)}\Big]\d s\Big\vert : t\in[0,b]\Big\rbrace \geq 2\delta\Big]\\
    &\qquad\qquad \leq\frac{d-1}{\delta}\,\underbrace{\EEE^{x}\Big[\!\int_{[0,b]}  a(s,\vert X_s-x\vert)\, \Big\vert\frac{1}{\rO_\varepsilon(X_s)} - \frac{1}{\rO(X_s)}\Big\vert\d s\Big]}_{\rmA_\varepsilon}\\
    &\qquad\qquad\qquad\qquad + \frac{1}{\delta}\,\underbrace{\EEE^{x}\Big[\!\int_{[0,b]}a(s,\vert X_s-x\vert)\,\frac{\varepsilon^2}{\rO_\varepsilon(X_s)^3}\d s\Big]}_{\rmB_\varepsilon}\!.
\end{align}
For later use, we define the positive constant
\begin{align}\label{Eq:ETADEF}
    \eta:=\frac{\vert x-o\vert}{2}.
\end{align}

\textit{Claim} A. We claim
\begin{align}
    \lim_{\varepsilon\to 0^+}\rmA_\varepsilon=0.
\end{align}
To this aim, we split up the expectation defining $\rmA_\varepsilon$ as
\begin{align}
    \rmA_\varepsilon &= \underbrace{\EEE^{x}\Big[\!\int_{[0,b]} a(s,\vert X_s-x\vert)\,\Big\vert\frac{1}{\rO_\varepsilon(X_s)} -\frac{1}{\rO(X_s)}\Big\vert\,\One_{\{\rO(X_s)\geq \eta\}}\d s\Big]}_{\rmI_\varepsilon}\\
    &\qquad\qquad + \underbrace{\EEE^{x}\Big[\!\int_{[0,b]} a(s,\vert X_s-x\vert)\,\Big\vert\frac{1}{\rO_\varepsilon(X_s)} -\frac{1}{\rO(X_s)}\Big\vert\,\One_{\{\rO(X_s)< \eta\}}\d s\Big]}_{\rmI\rmI_\varepsilon}\!.
\end{align}

\textit{Part} A.I. Given $s\in [0,b]$, on the event $\{\rO(X_s)\geq \eta\}$ we clearly have
\begin{align}
    \Big\vert\frac{1}{\rO_\varepsilon(X_s)} -\frac{1}{\rO(X_s)}\Big\vert \leq \frac{2}{\eta}.
\end{align}
Thus, by \eqref{Eq:Integrability} and Lebesgue's dominated convergence theorem,
\begin{align}
    \lim_{\varepsilon\to 0^+} \rmI_\varepsilon=0.
\end{align}

\textit{Part} A.II. Since
\begin{align}
    \rmI\rmI_\varepsilon \leq \EEE^{x}\Big[\!\int_{[0,b]} a(s,\vert X_s-x\vert)\,\frac{1}{\rO(X_s)} \,\One_{\{\rO(X_s)< \eta\}}\d s\Big]
\end{align}
from a pointwise inequality between the respective integrands, by Lebesgue's dominated convergence theorem it suffices to prove the latter expectation is finite to establish
\begin{align}\label{Eq:concludeproof}
    \lim_{\varepsilon\to 0^+}\rmI\rmI_\varepsilon=0.
\end{align}
As in the proof of \cref{Th:CenteredRadial},
\begin{align}
    &\EEE^{x}\Big[\!\int_{[0,b]} a(s,\vert X_s-x\vert)\,\frac{1}{\rO(X_s)} \,\One_{\{\rO(X_s)< \eta\}}\d s\Big]\\
    &\qquad\qquad = \int_{[0,b]}\int_{B_{R(s)}(x)\cap B_\eta(o)} a(s,\vert y-x\vert)\,\frac{1}{\rO(y)}\,U(s,\vert y-x\vert)\d y\d s.
\end{align}
The triangle inequality gives $\vert y - x\vert > \eta$ for every $\smash{y\in B_\eta(o)}$. Moreover, by definition of $U$ we have $a(s,r)\,U(s,r) = 0$ for every $s\in[0,R^{-1}(\eta)]$ as long as $r \geq \eta$, where $R^{-1}$ is the inverse of $R$ from \eqref{Eq:R-1def}. Moreover, as $R$ is bounded on $[0,b]$, so is $\vert \cdot-\,x\vert$ on $B_{R(b)}(x)$. These observations combine with \cref{Le:aU-bounds-m} to imply
\begin{align}\label{Eq:Analog}
\begin{split}
    &\int_{[0,b]}\int_{B_{R(s)}(x)\cap B_\eta(o)} a(s,\vert y-x\vert)\,\frac{1}{\rO(y)}\,U(s,\vert y-x\vert)\d y\d s\\
    &\qquad\qquad \leq \int_{[R^{-1}(\eta),b]}\int_{B_{R(b)}(x)\cap B_\eta(o)} a(s,\vert y-x\vert)\,\frac{1}{\rO(y)}\,U(s,\vert y-x\vert)\d y\d s\\
    &\qquad\qquad \lesssim \Big[\!\int_{[R^{-1}(\eta),b]} s^{-1-d/\beta+p/(p-1)\beta}\d s\Big]\,\Big[\!\int_{B_\eta(o)} \frac{1}{\rO(y)}\d y\Big]\\
    &\qquad\qquad \lesssim_{b,\eta}  \int_{[0,\eta]} r^{d-2}\d r,
    \end{split}
\end{align}
where we applied polar coordinates in the last inequality. Since $d$ is at least two, the last integral is clearly finite, to conclude the proof of \eqref{Eq:concludeproof}.

\textit{Claim} B. We claim
\begin{align}
    \lim_{\varepsilon\to 0^+}\rmB_\varepsilon=0.
\end{align}
By an analogous decomposition argument as for Claim A, it suffices to prove
\begin{align}
    \lim_{\varepsilon\to 0^+}\EEE^{x}\Big[\!\int_{[0,b]} a(s,\vert X_s-x\vert)\,\frac{\varepsilon^2}{\rO_\varepsilon(X_s)^3}\,\One_{\{\rO(X_s)<\eta\}}\d s\Big]=0;
\end{align}
in turn, in analogy to \eqref{Eq:Analog} this reduces to showing
\begin{align}
    \lim_{\varepsilon\to 0^+} \int_{[0,\eta]} \frac{\varepsilon^2\,r^{d-1}}{\sqrt{r^2+\varepsilon^2}^3}\d r=0.
\end{align}
When $\varepsilon$ is sufficiently small, we split up the integral in question as
\begin{align}
    \int_{[0,\eta]} \frac{\varepsilon^2\,r^{d-1}}{\sqrt{r^2+\varepsilon^2}^3}\d r =\int_{[0,\varepsilon]} \frac{\varepsilon^2\,r^{d-1}}{\sqrt{r^2+\varepsilon^2}^3}\d r + \int_{[\varepsilon,\eta]} \frac{\varepsilon^2\,r^{d-1}}{\sqrt{r^2+\varepsilon^2}^3}\d r.
\end{align}
For the first summand, we use the inequality
\begin{align}
    \sqrt{r^2+\varepsilon^2}^3\geq \varepsilon^3
\end{align}
for every $r\in\R$ and obtain
\begin{align}
    \int_{[0,\varepsilon]} \frac{\varepsilon^2\,r^{d-1}}{\sqrt{r^2+\varepsilon^2}^3}\d r\leq \frac{1}{\varepsilon}\int_{[0,\varepsilon]}r^{d-1}\d r \lesssim \varepsilon^{d-1},
\end{align}
which converges to zero as $\varepsilon\to 0^+$ since $d$ is no less than two.
For the second summand, we employ the inequality
\begin{align}
    \sqrt{r^2+\varepsilon^2}^3\geq r^3
\end{align}
for every $r\in\R_+$ instead, which entails
\begin{align}
    \int_{[\varepsilon,\eta]} \frac{\varepsilon^2\,r^{d-1}}{\sqrt{r^2+\varepsilon^2}^3}\d r \leq \varepsilon^2\int_{[\varepsilon,\eta]} r^{d-4}\d r \lesssim_\eta  \begin{cases}
        \varepsilon & \textnormal{if }d=2,\\
        \varepsilon^2 + \varepsilon^2\,\big\vert\!\log\varepsilon\big\vert & \textnormal{if }d=3,\\
        \varepsilon^2 & \textnormal{otherwise}.
    \end{cases}
\end{align}
All these terms converge to zero as $\varepsilon\to 0^+$. This completes the proof.
\end{proof}

\subsection{Uncentered case, dimension one}

\begin{theorem}[Tanaka--Meyer formula for one-dimensional radial process]\label{Th:RadialOne} Let $d=1$. For every $x\in\R\setminus\{o\}$, there exist a standard $\smash{\scrF_\bullet}$-Brownian motion $B$ and a nonnegative, nondecreasing, and $\smash{\scrF_\bullet}$-adapted process $L^o$ with locally bounded variation such that
\begin{align}
    \rmd \rO(X_t) &= \sqrt{2a(t,\vert X_t-x\vert)}\d B_t + a_r(t,\vert X_t-x\vert)\sgn(X_t-x)\sgn(X_t-o)\d t  + \rmd L_t^o\\
    &=\sqrt{2a(t,\vert X_t-x\vert)}\d B_t + \sfL_{p,m}^x\rO(X_t)\d t + \rmd L_t^o\quad\PPP^{x}\textnormal{\textit{-a.s.}},
\end{align}
where $\smash{\sfL_{p,m}^x}$ is from \eqref{Eq:Generator}.
\end{theorem}

The above process $L^o$ is precisely the local time associated with the semimartingale $\rO\circ X$, cf.~\cref{Cor:Semim}. In fact, the following proofs use classical one-dimensional theory of local times for semimartingales, see Revuz--Yor \cite{revuz-yor1999}*{§VI} for details.

\begin{proof}[Proof of \cref{Th:RadialOne}] By \eqref{Eq:ClosedMcKeanVlasov} and \eqref{Eq:Integrability}, $X$ is a continuous semimartingale. Tanaka's formula \cite{revuz-yor1999}*{Thm.~VI.1.2},  \eqref{Eq:ClosedMcKeanVlasov}, and the proof of \cref{Pr:Threeucp} imply that given $o\in\R$ there is a nonnegative, nondecreasing, and $\smash{\scrF_\bullet}$-adapted process $L^o$ with  locally bounded variation such that
\begin{align}
    \rmd \rO(X_t) &= \sgn(X_t-o)\d X_t +\rmd L_t^o\\
    &= \sqrt{2a(t,\vert X_t-x\vert)}\d B_t + a_r(t,\vert X_t-x\vert)\sgn(X_t-x)\sgn(X_t-o)\d t\\
    &\qquad\qquad+\rmd L_t^o \quad\PPP^{x}\textnormal{-a.s.}
\end{align}
This is the desired statement.
\end{proof}

In the sequel, we will fix a modification of the local time as follows.

\begin{theorem}[Properties and nontriviality of local time]\label{Th:Modif} Let $d=1$. For every $x\in\R$, any process $L$ as appearing in \cref{Th:RadialOne} admits a nonrelabeled modification that obeys the following properties for every $o\in\R$.
\begin{enumerate}[label=\textnormal{\textcolor{black}{(}\roman*\textcolor{black}{)}}]
    \item\label{La:i} Outside a $\PPP^{x}$-negligible set, $L_t^o$ depends jointly continuously on $\smash{(t,o)\in\R_+\times\R}$.
    \item\label{La:iv} There exists a $\smash{\PPP^{x}}$-negligible set outside which for every $t\in\R_+$, we have the asymptotic occupation time  identity
    \begin{align}
        L_t^o &=\lim_{\varepsilon\to 0^+}\frac{2}{\varepsilon}\int_{[0,t]}1_{\{o\leq X_s\leq o+\varepsilon\}}\,a(s,\vert X_s-x\vert)\d s.
    \end{align}
    \item\label{La:v} We have
    \begin{align}
        L^o = \PPP^{x}\textnormal{-ucp-}\!\!\!\lim_{\varepsilon\to 0^+}\int_{[0,\bullet]} a(s,\vert X_s-x\vert)\,\frac{\varepsilon^2}{\rO_\varepsilon(X_s)^3}\d s.
    \end{align}
    \item\label{La:vi} For every $t\in \R_+$,
    \begin{align}
        \EEE^{x}\big[L_t^o\big] = 2\int_{[0,t]} a(s,\vert o-x\vert)\,U(s,\vert o-x\vert)\d s.
    \end{align}
\end{enumerate}
\end{theorem}
\begin{proof} By \cref{Th:RadialOne} and Revuz--Yor \cite{revuz-yor1999}*{Thm.~VI.1.7}, $L$ has a modification obeying \ref{La:i}, which we  fix in the remainder of this proof.

The asymptotic occupation time identity \ref{La:iv} is a consequence of \cite{revuz-yor1999}*{Cor.~VI.1.9}.

By the above Itô expansion of $\rO\circ X$, \cref{Pr:Threeucp}, \cref{Le:Itoeps1}, and \cref{Cor:Itoeps}, we necessarily have \ref{La:v}.

It remains to establish \ref{La:vi}. Thanks to the occupation time formula \cite{revuz-yor1999}*{Cor.~VI.1.6} and \eqref{Eq:ClosedMcKeanVlasov}, there exists a $\smash{\PPP^{x}}$-negligible set outside which every $t\in\R_+$ and every Borel measurable function $\varphi\colon \R\to\R_+$ satisfy
\begin{align}
    \int_{\R}\varphi(o)\,L_t^o\d o = \int_{[0,t]}\varphi(X_s)\d [X,X]_s = 2\int_{[0,t]}\varphi(X_s)\,a(s,\vert X_s-x\vert)\d s.
\end{align}
Taking expectations and using Fubini's theorem,
\begin{align}
    \int_{\R}\varphi(o)\,\EEE^{x}\big[L_t^o\big]\d o &= 2\int_{[0,t]} \EEE^{x}\big[\varphi(X_s)\,a(s,\vert X_s-x\vert)\big]\d s\\
    &= 2 \int_{[0,t]} \int_{\R}\varphi(o)\,a(s,\vert o-x\vert)\,U(s,\vert o-x\vert)\d o\d s\\
    &= 2 \int_{\R}\varphi(o) \int_{[0,t]} a(s,\vert o-x\vert)\,U(s,\vert o-x\vert)\d s\d o.
\end{align}
By \ref{La:i}, the expectation $\smash{\EEE^{x}[L_t^o]}$ depends continuously on $o\in\R$; \cref{Le:CtyG} below shows the inner integral on the right-hand side depends continuously on $o\in\R$. Therefore,  the claim \ref{La:vi} follows from the arbitrariness of $\varphi$.
\end{proof}

\begin{lemma}[Continuity]\label{Le:CtyG} Let $d=1$ and $t\in\R_+$. Define $G\colon\R_+\to\R_+$ by
\begin{align}
    G(r) := \int_{[0,t]} a(s,r)\,U(s,r)\d s.
\end{align}
Then $G$ is continuous.
\end{lemma}

\begin{proof} Recall from \cref{Le:aU-bounds-m} that for every $s\in (0,t]$ and every $r\in\R_+$,
\begin{align}
    a(s,r)\,U(s,r) \lesssim s^{-1-1/\beta+p/(p-1)\beta}\,r^{(p-2)/(p-1)}.
\end{align}
By our hypothesis \eqref{Eq:ranges}, we necessarily have $p>2$ in one dimension. In turn, since
\begin{align}
    -1-\frac{1}{\beta}+\frac{p}{(p-1)\beta} = -1 - \frac{1}{\beta} + \frac{p-2}{(p-1)\beta} + \frac{2}{(p-1)\beta} >-1
\end{align}
thanks to \eqref{Eq:Expd2} and since the above exponent $(p-2)/(p-1)$ is positive, we have found a dominating function for $a(\cdot,r)\,U(\cdot,r)$ that is locally uniformly bounded in $r\in\R_+$. The claim follows from Lebesgue's dominated convergence theorem.
\end{proof}

In particular, this lemma and \cref{Th:Modif} confirm the triviality of the local time in \cref{Th:CenteredRadial} in the one-dimensional case because for every $t\in\R_+$,
\begin{align}\label{Eq:Contrast}
    \EEE^{o}\big[L_t^o\big] = 2\int_{[0,t]} a(s,0)\,U(s,0)\d s =0.
\end{align}
This vanishing manifests a classical mechanism: by the occupation time formula, the local time at a point where the diffusion coefficient degenerates continuously is necessarily trivial; systematic criteria for one-dimensional stochastic differential equations with such singular points go back to Engelbert--Schmidt and are surveyed by Cherny--Engelbert \cite{cherny-engelbert2005}. Specific to our setting is the precise rate $\smash{r^{(p-2)/(p-1)}}$ at which the diffusivity degenerates at the center, cf.~\cref{Le:aU-bounds-m}, which is positive in dimension one by \eqref{Eq:ranges}.

In turn, \cref{Le:aU-bounds-m} immediately gives the following.

\begin{corollary}[Asymptotics for expected local time]\label{Cor:ExpLocalTime} Let $d=1$. Then for every $\theta>1$, every $x,o\in\R$, and every $\smash{t\in [\theta R^{-1}(\vert o-x\vert),\infty)}$, the local time $L^o$ satisfies
\begin{align}
    t^{-1/\beta +p/(p-1)\beta}\,\vert o-x\vert^{(p-2)/(p-1)} \lesssim_\theta
    \EEE^{x}\big[L_t^o\big]
    \lesssim t^{-1/\beta +p/(p-1)\beta}\,\vert o-x\vert^{(p-2)/(p-1)};
\end{align}
the upper bound holds in fact for every $t\in\R_+$.
\end{corollary}

On the other hand, if $R(t)\leq \vert o-x\vert$ above, then clearly
\begin{align*}
    \EEE^x\big[L_t^o\big]=0
\end{align*}
since in this range of $t$, the process $X$ starting at $x$ cannot have hit $o$ yet. Since $\smash{\EEE^x[L_t^o]}$ depends continuously on $\smash{t\in\R_+}$, no lower bound as in \cref{Cor:ExpLocalTime} can hold uniformly for times near $\smash{R^{-1}(\vert o-x\vert)}$; this explains the constraint $\smash{t\geq\theta R^{-1}(\vert o-x\vert)}$.

As the local time $L^o$ from \cref{Th:RadialOne} is $\PPP^{x}$-a.s.~continuous and nondecreasing in time for every $x\in\R\setminus \{o\}$, it is the primitive of a uniquely determined nontrivial random measure $\smash{\mu^o}$ on $\R_+$. Applying Revuz--Yor \cite{revuz-yor1999}*{Prop.~VI.1.3} then implies the following.

\begin{proposition}[Jumps of local time]\label{Pr:JumpsLocalTime} For every $x\in\R\setminus\{o\}$, the random measure $\smash{\mu^o}$ is $\PPP^{x}$-a.s.~concentrated on $\{t\in\R_+: X_t=o\}$.
\end{proposition}

\subsection{Summary and comparison with Euclidean Brownian motion} We  now provide a compact representation of the Tanaka--Meyer semi\-martingale formula of the radial process $\smash{\rO\circ X}$ we have established above.

\begin{theorem}[Tanaka--Meyer formula for radial process]\label{Th:Summary} Given an initial point $\smash{x\in\R^d}$, define $\smash{\widetilde{L}\colon \R_+\times \R^d\times\Omega\to\R_+}$ by
\begin{align*}
    \widetilde{L}_t^o(\omega):= \begin{cases}
        0 & \textnormal{\textit{if} }d \geq 2 \textnormal{ \textit{or} }o=x,\\
        \textnormal{\textit{the modification from \cref{Th:Modif}}} & \textnormal{\textit{otherwise}}.
    \end{cases}
\end{align*}
Then there exists a standard $\scrF_\bullet$-Brownian motion $B$ such that
\begin{align*}
    \rmd \rO(X_t) &= \sqrt{2a(t,\vert X_t-x\vert)}\d B_t + a(t,\vert X_t-x\vert)\,\frac{d-1}{\rO(X_t)}\d t\\
    &\qquad\qquad + a_r(t,\vert X_t-x\vert)\sgn(X_t-x)\cdot\sgn(X_t-o)\d t + \rmd\widetilde{L}_t^o\quad\PPP^x\textnormal{\textit{-a.s.}}
\end{align*}
\end{theorem}

\begin{corollary}[Semimartingale property of radial process]\label{Cor:Semim} For every $\smash{x\in\R^d}$, the radial process $\rO\circ X$ is a semimartingale under $\PPP^{x}$.
\end{corollary}

The above theorem exhibits analogies but also remarkable differences to properties of radial processes of Euclidean Brownian motion we briefly recall now.

\begin{remark}[Comparison with Euclidean Brownian motion]\label{Re:BMLocal} Let $W$ be a standard $d$-dimensional Brownian motion (on a stochastic basis filtered by $\scrF_\bullet$) and let $\PPP^o$ be the path law of $\smash{X := o+\sqrt{2}\,W}$. If $d$ is no less than two, it is a standard fact in stochastic calculus, cf.~e.g.~Revuz--Yor \cite{revuz-yor1999}*{§§VI.3, XI.1}, to verify the radial process $\rO\circ X$ is a rescaled Bessel process: there is a standard $\scrF_\bullet$-Brownian motion $B$ on $\R$ satisfying
\begin{align}
    \rmd B_t = \sgn(X_t-o)\cdot \rmd W_t\quad\PPP^{o}\textnormal{-a.s.}
\end{align}
such that
\begin{align}
    \rmd \rO(X_t) = \sqrt{2}\,\rmd B_t +\frac{d-1}{\rO(X_t)}\d t = \sqrt{2}\,\rmd B_t +\Delta \rO(X_t)\d t\quad \PPP^{o}\textnormal{-a.s.}
\end{align}
This is the formal version of \cref{Th:CenteredRadial} for Euclidean Brownian motion, which we do not include in our hypothesis \eqref{Eq:ranges}. In particular, local time does not occur in dimensions at least two, both for the process we study and Euclidean Brownian motion.

However, in the one-dimensional situation the conclusion of \cref{Th:CenteredRadial} differs from Euclidean Brownian motion when $x$ and $o$ coincide. Here, the Tanaka--Meyer formula for $\rO\circ X$ is formally the one from \cref{Th:RadialOne} for Euclidean Brownian motion, viz.
\begin{align}
    \rmd \rO(X_t)  = \sqrt{2}\d B_t + \rmd L_t^o\quad\PPP^{o}\textnormal{-a.s.}
\end{align}
However, its local time $L^o$ is nontrivial also under $\PPP^o$. In fact, \cite{revuz-yor1999}*{Cor.~VI.2.4} states
\begin{align}
    \PPP^o\big[L_\infty^o=\infty\big] = 1.
\end{align}
Moreover, similarly to item \ref{La:vi} in \cref{Th:Modif} one computes
\begin{align}
    \EEE^o\big[L_t^o\big] = 2\int_{[0,t]}\frac{1}{\sqrt{4\pi s}}\d s= 2\sqrt{\frac{t}{\pi}},
\end{align}
which contrasts \eqref{Eq:Contrast}.
\end{remark}

\subsection{Applications}\label{Sub:Exittime} Throughout the subsection, we fix $\smash{o\in\R^d}$ and consider the distance function $\smash{\rO}$ from \eqref{Eq:Distance}. Fix a stochastic basis $\smash{(\Omega,\scrF,\scrF_\bullet,\PPP^x)}$ as above, where $\smash{x\in\R^d}$. Given $R>0$, we denote the first exit time from the ball $B_R(o)$ by 
\begin{align}\label{Eq:tauR}
    \tau_R := \inf\{t\in\R_+ : X_t\notin B_R(o)\}.
\end{align}
We also recall the Barenblatt radius $R$ from \eqref{Eq:Rdef}: there exists an explicit constant $\rho>0$ depending only on $d$, $p$, and $m$ such that for every $t\in\R_+$,
\begin{align}\label{Eq:Rrho}
    R(t) = \rho\,t^{1/\beta}.
\end{align}

Our first result is the following.

\begin{theorem}[Expectation of exit time]\label{Th:Expexittime} For every $R>0$ and every $\alpha\in(0,d/\beta)$,
\begin{align*}
    \EEE^o\big[\tau_R^\alpha\big] \asymp_\alpha R^{\alpha\beta}.
\end{align*}
\end{theorem}

\begin{remark}[Linear expectation] In particular, if $d>\beta$ the previous theorem yields a precise asymptotic for the genuine expected exit time $\smash{\EEE^o[\tau_R]}$ by $\smash{R^\beta}$. In the case $m=1$ of the $p$-Laplace equation \eqref{Eq:pLaplaceequation}, $d>\beta$ holds if and only if $p < 3d/(d+1)$.
\end{remark}

We start with a deterministic lower bound on $\tau_R$ reflecting the finite speed of propagation of the Barenblatt support.

\begin{proposition}[Exit time lower bound]\label{Pr:DetExit} For every $R>0$ and every $\smash{x\in B_R(o)}$,
\begin{align}\label{Eq:taulowerbound}
    \tau_R\geq \rho^{-\beta}\,\big[R-\vert x-o\vert\big]^{\beta}\quad\PPP^{x}\textnormal{\textit{-a.s.}}
\end{align}
\end{proposition}

\begin{proof} Recall for every $t\in\R_+$ the law of $X_t$ under $\smash{\PPP^x}$ is supported in $\smash{\overline{B}_{R(t)}(x)}$; equivalently, for every \emph{fixed} $t>0$, we have $\vert X_t-x\vert \leq R(t)$ $\PPP^x$-a.s. Since $X$ has continuous paths and $R$ is continuous, this implies $\PPP^x$-a.s.~the previous inequality holds for every $t\in\R_+$.

Therefore, fix an $\omega\in\Omega$ such that $\smash{X_\bullet(\omega)\colon \R_+\to\R^d}$ is continuous and the inequality $\smash{\vert X_t(\omega)-x\vert\leq \rho\,t^{1/\beta}}$ holds for every $t\in\R_+$. We claim \eqref{Eq:taulowerbound} holds at $\omega$. There is nothing to show if $\tau_R(\omega)=\infty$; thus, assume $\tau_R(\omega)$ is finite. By definition of $\tau_R$ and continuity of $X_\bullet(\omega)$, we have $\smash{X_{\tau_R(\omega)}(\omega)\in\partial B_R(o)}$. By the triangle inequality and the above estimate,
\begin{align*}
    R = \rO(X_{\tau_R(\omega)}(\omega))\leq \big\vert X_{\tau_R(\omega)}(\omega)-x\big\vert + \vert x-o\vert \leq \rho\,\tau_R(\omega)^{1/\beta}+\vert x-o\vert.
\end{align*}
Since $\vert x-o\vert < R$, rearranging terms yields the claim.
\end{proof}

In particular, \cref{Pr:DetExit} provides the sharp deterministic lower bound
\begin{align}\label{Eq:tauRlowerbound}
    \tau_R\geq \rho^{-\beta}\,R^\beta\quad\PPP^o\textnormal{-a.s.};
\end{align}
in particular, for every $\smash{t\in[0,\rho^{-\beta}R^\beta]}$,
\begin{align}\label{Eq:tausmalltimes}
    \PPP^o\big[\tau_R>t\big] = 1.
\end{align}
The next result is a complementary upper bound on the survival probability (when the process starts at the center $o$). It shows the exit from $B_R(o)$ is in fact polynomially fast in time. To this aim, we recall the constants $\gamma$ and $C$ from \eqref{Eq:betagammakappa} and \eqref{Eq:Normal}.

\begin{proposition}[Polynomial survival bound]\label{Th:Exittime} For every $R>0$ and every $t>0$,
\begin{align}
    \PPP^o\big[\tau_R>t\big] \leq 1\wedge\frac{\omega_{d-1}C^\gamma}{d}\,t^{-d/\beta}\,R^d.
\end{align}
\end{proposition}

\begin{proof} The bound $\smash{\PPP^o[\tau_R>t]\leq 1}$ is trivial. 

To show the nontrivial bound, note the event $\smash{\{\tau_R>t\}}$ is contained in $\smash{\{X_t\in B_R(o)\}}$. By definition of the law of $X_t$ and using polar coordinates,
\begin{align}
    \PPP^o\big[\tau_R>t\big] \leq \PPP^o\big[X_t\in B_R(o)\big] = \omega_{d-1}\!\int_{[0,R]} U(t,r)\,r^{d-1}\d r.
\end{align}
Using the inequality $\smash{U(t,r)\leq C^{\gamma}\,t^{-d/\beta}}$ for every $r\in \R_+$ implied by \eqref{Eq:US} gives
\begin{align*}
    \int_{[0,R]} U(t,r)\,r^{d-1}\d r \leq C^\gamma\,t^{-d/\beta} \int_{[0,R]} r^{d-1}\d r= \frac{C^\gamma}{d}\,t^{-d/\beta}\,R^d.
\end{align*}
Combining these estimates yields the claim.
\end{proof}

\begin{proof}[Proof of \cref{Th:Expexittime}] By the deterministic lower bound \eqref{Eq:tauRlowerbound}, for every $\alpha> 0$ we get
\begin{align*}
    \EEE^o\big[\tau_R^\alpha\big] \gtrsim_\alpha R^{\alpha\beta}.
\end{align*}

We turn to the upper bound. Cavalieri's formula and \eqref{Eq:tausmalltimes} imply
\begin{align*}
    \EEE^o\big[\tau_R^\alpha\big] &= \alpha\int_{\R_+} t^{\alpha-1}\,\PPP^o\big[\tau_R>t\big]\d t\\
    &= \alpha\int_{[0,\rho^{-\beta}R^\beta]} t^{\alpha-1}\d t + \alpha\int_{[\rho^{-\beta}R^{\beta},\infty)} t^{\alpha-1}\,\PPP^o\big[\tau_R>t\big]\d t.
\end{align*}
Since $\alpha > 0$, the first summand on the right-hand side easily gives
\begin{align*}
    \int_{[0,\rho^{-\beta}R^\beta]} t^{\alpha-1}\d t \asymp_\alpha R^{\alpha\beta}.
\end{align*}
For the second summand, we use \cref{Th:Exittime} to obtain
\begin{align*}
    \int_{[\rho^{-\beta}R^{\beta},\infty)} t^{\alpha-1}\,\PPP^o\big[\tau_R>t\big]\d t \lesssim R^d\int_{[\rho^{-\beta}R^\beta,\infty)} t^{\alpha-1-d/\beta}\d t \asymp_\alpha R^{\alpha\beta};
\end{align*}
in the last identity, we used the hypothesis $\alpha < d/\beta$, which makes the latter integrand integrable at infinity. This yields the claim.
\end{proof}

Next, we show that, beyond the deterministic ceiling \eqref{Eq:Rrho}, the self-similarly rescaled radial process admits precise control in expectation. Consider the process $Y$ given by
\begin{align}\label{Eq:Ydef}
    Y_t := \frac{\rO(X_t)}{t^{1/\beta}}.
\end{align}
By \eqref{Eq:Rrho} and the proof of \cref{Pr:DetExit}, $\PPP^o$-a.s.~every $t>0$ satisfies
\begin{align}\label{Eq:Yrho}
    Y_t\leq \rho.
\end{align}

\begin{proposition}[Law of rescaled process]\label{Pr:LawResc} For every $t>0$, the law of $Y_t$ under $\PPP^o$ is equal to $V\,\Leb^1$, where $V\colon \R\to\R_+$ is defined by
\begin{align*}
    V(r) :=\begin{cases} \omega_{d-1}\,\big[C - \kappa\,r^{p/(p-1)}\big]_+^\gamma\,r^{d-1} & \textnormal{if } r\in\R_+,\\
        0 & \textnormal{otherwise};
    \end{cases}
\end{align*}
in particular, it does not depend on $t$.
\end{proposition}

\begin{proof} Clearly, the sought law is concentrated on $\R_+$. Let $A\subset\R_+$ be Borel measurable. Using the law of $X_t$ from \eqref{Eq:OriginalMCV}, polar coordinates, and the form of $U$ from \eqref{Eq:U},
\begin{align*}
    \PPP^o\big[Y_t \in A\big] &= \PPP^o\big[\rO(X_t)\in t^{1/\beta}A\big]\\
    &= \int_{\R^d} 1_{t^{1/\beta}A}(\vert y-o\vert)\,U(t,\vert y-o\vert)\d y\\
    &= \omega_{d-1}\int_{t^{1/\beta}A} U(t,r)\,r^{d-1}\d r\\
    &= \omega_{d-1}\, t^{-d/\beta}\int_{t^{1/\beta}A} \big[C-\kappa\big[t^{-1/\beta}\,r\big]^{p/(p-1)}\big]_+^\gamma\,r^{d-1}\d r.
\end{align*}
By the substitution $r = t^{1/\beta}\,s$ and $\d r = t^{1/\beta}\d s$, this equals
\begin{align*}
t^{-d/\beta + (d-1)/\beta + 1/\beta}\int_{A} V(s)\d s = \int_A V(s)\d s.
\end{align*}
By the arbitrariness of $A$, the claim follows.
\end{proof}

\begin{remark}[Alternative representations of Barenblatt profiles]\label{Re:AltRep} The time-invariance from \cref{Pr:LawResc} manifests the self-similarity of the Barenblatt family and, as such, is not tied to the McKean--Vlasov dynamics \eqref{Eq:OriginalMCV}. Indeed, De Gregorio \cite{degregorio2018} and De Gregorio--Garra \cite{degregorio-garra2020} represent Barenblatt-type profiles --- including those of the parabolic $p$-Laplace equation \eqref{Eq:pLaplaceequation} --- by random flights and random velocity models, and observe the analogous time-invariance of the rescaled radial law, cf.~\cite{degregorio2018}*{Prop.~3.2}. These processes are designed to match a given profile and are unrelated to \eqref{Eq:OriginalMCV}; in particular, no pathwise decomposition in the spirit of \cref{Th:CenteredRadial} is available for them.
\end{remark}

Under $\PPP^o$, \cref{Th:CenteredRadial} exhibits the radial process as the semimartingale
\begin{align}\label{Eq:rXdecomp}
    \rmd \rO(X_t) = \sqrt{2a(t,\rO(X_t))}\d B_t + b(t,\rO(X_t))\d t\quad\PPP^o\textnormal{-a.s.},
\end{align}
where the drift density $b$ is given by
\begin{align}\label{Eq:bdef}
    b(t,\rO(X_t)) := a(t,\rO(X_t))\,\frac{d-1}{\rO(X_t)} + a_r(t,\rO(X_t))\quad\PPP^o\textnormal{-a.s.}
\end{align}
Since the diffusivity $a$ degenerates at the center, cf.~\cref{Le:aU-bounds-m}, the drift density $b$ does not admit a deterministic bound. The following lemma shows its expectation, and the one of the diffusivity along the process, nevertheless obeys the natural self-similar scaling. It solely relies on the marginal laws prescribed by \eqref{Eq:ClosedMcKeanVlasov}. These bounds are implicit in the proof of Barbu--Grube--Rehmeier--Röckner \cite{barbu-grube-rehmeier-rockner2025+}*{Thm. 4.4}, where the corresponding double integrals (over subsets of $\R_+$ and $\Omega$) are shown to be finite in order to verify the integrability conditions of \cref{Def:PWS} in \cref{Th:ExistencePWS}.

\begin{lemma}[Coefficient bounds in expectation]\label{Le:Coeffq} For every $t>0$,
\begin{align*}
    \EEE^o\big[a(t,\rO(X_t))\big] &\lesssim t^{-1+2/\beta},\\
    \EEE^o\big[\big\vert b(t,\rO(X_t))\big\vert\big] &\lesssim t^{-1+1/\beta}.
\end{align*}
\end{lemma}

\begin{proof} We recall the Barenblatt radius $R$ from \eqref{Eq:Rrho}. Since the law \eqref{Eq:OriginalMCV} of $X_t$ under $\PPP^o$ is supported on $\smash{\overline{B}_{R(t)}(o)}$, another application of polar coordinates gives for every Borel measurable function $\varphi\colon (0,R(t))\to \R_+$
\begin{align}\label{Eq:PolarExp}
    \EEE^o\big[\varphi\circ \rO(X_t)\big] = \omega_{d-1}\int_{[0,R(t)]}\varphi(r)\,U(t,r)\,r^{d-1}\d r.
\end{align}
We will employ this together with the following inequalities, where the first follows from  \cref{Le:aU-bounds-m} and the other two are trivial: for every $r\in (0,R(t))$ and every $\theta\in \R\setminus\{0\}$,
\begin{align}\label{Eq:aaaaaaaaaabounddd}
\begin{split}
    a(t,r)\,U(t,r) &\lesssim t^{-1-d/\beta+p/(p-1)\beta}\,r^{(p-2)/(p-1)},\\
    U(t,r) &\lesssim t^{-d/\beta},\\
    R(t)^{\theta} &\asymp_\theta t^{\theta/\beta}.
    \end{split}
\end{align}

We now show the first claimed estimate. By \eqref{Eq:PolarExp} and \eqref{Eq:aaaaaaaaaabounddd},
\begin{align}
    \EEE^o\big[a(t,\rO(X_t))\big] &= \omega_{d-1} \int_{[0,R(t)]} a(t,r)\,U(t,r)\,r^{d-1}\d r\\
    &\lesssim t^{-1-d/\beta+p/(p-1)\beta}\int_{[0,R(t)]} r^{d-1+(p-2)/(p-1)}\d r.
\end{align}
By \eqref{Eq:Expd2}, the exponent of $r$ is positive. Thus,
\begin{align}
    \EEE^o\big[a(t,\rO(X_t))\big]\lesssim t^{-1-d/\beta+p/(p-1)\beta}\,R(t)^{d+(p-2)/(p-1)} \asymp t^{-1+2/\beta},
\end{align}
where we used \eqref{Eq:aaaaaaaaaabounddd} and the exponent identity
\begin{align}
    -\frac{d}{\beta}+\frac{p}{(p-1)\beta} + \frac{d}{\beta} + \frac{p-2}{(p-1)\beta} = \frac{2}{\beta}.
\end{align}

Finally, we address the second claimed estimate. By \eqref{Eq:bdef}, it suffices to estimate the expectations of $\smash{a(t,\rO(X_t))/\rO(X_t)}$ and of $\vert a_r(t,\rO(X_t))\vert$ separately. For the first term, \eqref{Eq:PolarExp} and the computation from \eqref{Eq:Larger!} yield
\begin{align}
    \EEE^o\Big[\frac{a(t,\rO(X_t))}{\rO(X_t)}\Big] &\lesssim t^{-1+1/\beta}.
    \end{align}
For the second term, recall from the proof of \cref{Le:aU-bounds-m} that 
\begin{align}
    \big\vert a_r(t,r)\big\vert \lesssim t^{-1+p/(p-1)\beta}\,r^{-1/(p-1)} + t^{-1}\,r.
\end{align}
Combining this with \eqref{Eq:PolarExp},
\begin{align}
    \EEE^o\big[\big\vert a_r(t,\rO(X_t))\big\vert\big] &= \omega_{d-1}\int_{[0,R(t)]} \big\vert a_r(t,r)\big\vert\,U(t,r)\,r^{d-1}\d r\\
    &\lesssim t^{-1-d/\beta+p/(p-1)\beta}\int_{[0,R(t)]} r^{d-1-1/(p-1)}\d r\\
    &\qquad\qquad + t^{-1-d/\beta}\int_{[0,R(t)]} r^d\d r.
\end{align}
By \eqref{Eq:Expd2}, the exponent in the second last integral is larger than $-1$. Thus, 
\begin{align}
    \EEE^o\big[\big\vert a_r(t,\rO(X_t))\big\vert\big] &\lesssim t^{-1-d/\beta+p/(p-1)\beta}\,R(t)^{d-1/(p-1)} + t^{-1-d/\beta}\,R(t)^{d+1}\asymp t^{-1+1/\beta},
\end{align}
where we used the exponent identities
\begin{align}
    -\frac{d}{\beta} + \frac{p}{(p-1)\beta} + \frac{d}{\beta} - \frac{1}{(p-1)\beta} = \frac{1}{\beta} = -\frac{d}{\beta}+\frac{d+1}{\beta}.
\end{align}
This completes the proof.
\end{proof}

\begin{theorem}[Stopped estimate for rescaled radial process]\label{Th:Stopped} There is a constant $c>0$ depending only on $p$, $m$, and $d$ with the following property. For every $s,t>0$ with $s<t$ and every $\scrF_\bullet$-stopping time $\tau$ with $\smash{s\leq \tau\leq t}$ $\smash{\PPP^o}$-a.s.,
\begin{align}
    \EEE^o\big[Y_\tau\big] + \frac{1}{\beta}\,\EEE^o\Big[\!\int_{[s,\tau]} \frac{Y_u}{u}\d u\Big] \leq \EEE^o\big[Y_s\big] + c\log\frac{t}{s}.
\end{align}
\end{theorem}

\begin{proof} Since $\smash{u^{-1/\beta}}$ depends smoothly on $u\in [s,\infty)$, the time-dependent Itô formula and \eqref{Eq:rXdecomp} imply $\PPP^o$-a.s.~on the given time interval that
\begin{align}\label{Eq:dY}
    \rmd Y_u = u^{-1/\beta}\,\sqrt{2a(u,\rO(X_u))}\d B_u + u^{-1/\beta}\,b(u,\rO(X_u))\d u - \frac{1}{\beta}\,\frac{Y_u}{u}\d u.
\end{align}
The second last Lebesgue integral is finite on $[s,t]$: by Fubini's theorem and \cref{Le:Coeffq},
\begin{align}\label{Eq:Driftint}
    \EEE^o\Big[\!\int_{[s,t]} u^{-1/\beta}\,\big\vert b(u,\rO(X_u))\big\vert\d u\Big] \lesssim \int_{[s,t]} u^{-1/\beta}\,u^{-1+1/\beta}\d u = \log\frac{t}{s}.
\end{align}

Define the process $N$ on $[s,t]$ by
\begin{align}
    N_v := \int_{[s,v]} u^{-1/\beta}\,\sqrt{2a(u,\rO(X_u))}\d B_u.
\end{align}
It is a continuous local martingale with $N_s=0$. By Fubini's theorem and \cref{Le:Coeffq},
\begin{align}
    \EEE^o\big[[N]_t\big] = 2\int_{[s,t]} u^{-2/\beta}\,\EEE^o\big[a(u,\rO(X_u))\big]\d u \lesssim \int_{[s,t]} u^{-2/\beta}\,u^{-1+2/\beta}\d u = \log\frac{t}{s},
\end{align}
which is finite. Consequently, $N$ is an $L^2$-bounded martingale. Since $\tau$ is bounded, the optional stopping theorem readily yields
\begin{align}\label{Eq:OptStop}
    \EEE^o\big[N_\tau\big] = 0.
\end{align}

Now we address the claim. Integrating \eqref{Eq:dY} on $[s,\tau]$ and rearranging terms gives
\begin{align}\label{Eq:Yint}
    Y_\tau + \frac{1}{\beta}\int_{[s,\tau]}\frac{Y_u}{u}\d u = Y_s + N_\tau + \int_{[s,\tau]} u^{-1/\beta}\,b(u,\rO(X_u))\d u\quad\PPP^o\textnormal{-a.s.}
\end{align}
All terms here are $\PPP^o$-integrable: for the left-hand side, this follows from \eqref{Eq:Yrho} and the inequality $\tau\leq t$, and for the right-hand side from \eqref{Eq:Yrho}, the martingale property of $N$, and  \eqref{Eq:Driftint}. Taking $\PPP^o$-expectations and using \eqref{Eq:OptStop} as well as
\begin{align}
    \EEE^o\Big[\!\int_{[s,\tau]} u^{-1/\beta}\,b(u,\rO(X_u))\d u\Big] \leq \EEE^o\Big[\!\int_{[s,t]} u^{-1/\beta}\,\big\vert b(u,\rO(X_u))\big\vert\d u\Big]
\end{align}
together with \eqref{Eq:Driftint} establishes the claim.
\end{proof}

The stopped estimate converts, by Markov's inequality, into a control on excursions of the radial process over moving self-similar boundaries.

\begin{corollary}[Exit estimate over a moving self-similar boundary]\label{Cor:MovBoundary} For every $s,t> 0$ with $s<t$ and every $\vartheta>0$,
\begin{align}\label{Eq:MovingExit} 
    \PPP^o\Big[\!\sup\!\Big\lbrace\frac{\rO(X_u)}{u^{1/\beta}} : u\in [s,t]\Big\rbrace\geq \vartheta\Big] \leq \frac{1}{\vartheta}\Big[\EEE^o\big[Y_s\big]+c\log\frac{t}{s}\Big],
\end{align}
where $c$ is the constant from \cref{Th:Stopped}.
\end{corollary}

\begin{proof} Define the $\scrF_\bullet$-stopping time
\begin{align}
    \tau := \inf\{u\in [s,\infty) : Y_u\geq \vartheta\}\wedge t;
\end{align}
it is a stopping time since $Y$ is continuous and $\scrF_\bullet$-adapted and $[\vartheta,\infty)$ is closed. Define
\begin{align*}
    E := \Big\lbrace\!\sup\!\Big\lbrace\frac{\rO(X_u)}{u^{1/\beta}} : u\in [s,t]\Big\rbrace\geq \vartheta\Big\rbrace.
\end{align*}
By continuity of the involved sample paths, the supremum of $Y$ over the compact interval $[s,t]$ is attained; hence on $E$ there is $u\in[s,t]$ with $Y_u\geq\vartheta$, which gives $\tau\leq u\leq t$ and, again by continuity, we also have $Y_\tau\geq\vartheta$. Since $Y$ is nonnegative, we thus obtain
\begin{align}
    Y_\tau \geq \vartheta\,\One_E\quad\PPP^o\textnormal{-a.s.}
\end{align}
Taking expectations with respect to $\PPP^o$ and applying \cref{Th:Stopped} yields
\begin{align*}
    \PPP^o\big[E\big] \leq \frac{1}{\vartheta}\,\EEE^o\big[Y_\tau\big] \leq \frac{1}{\vartheta}\Big[\EEE^o\big[Y_s\big] + c\log\frac{t}{s}\Big],
\end{align*}
which concludes the proof.
\end{proof}

\begin{remark}[Range of nontriviality]\label{Re:Nontrivial} By \eqref{Eq:Yrho}, the left-hand side of \eqref{Eq:MovingExit} vanishes for every $\vartheta>\rho$. On the other hand, by \cref{Pr:LawResc} we see the law of $Y_s$ with respect to $\PPP^o$ is independent of $s > 0$ and not concentrated at $\rho$, which implies that the constant expectation $\smash{\eta := \EEE^o[Y_s]}$ satisfies $\smash{\eta<\rho}$. Thus, \cref{Cor:MovBoundary} is nontrivial precisely for $\smash{\vartheta\in (\eta + c\log t/s,\rho)}$, a nonempty range when $t < s\,\rme^{(\rho-\eta)/c}$: it controls the excursions of the rescaled radial process over levels \emph{below} the Barenblatt boundary on logarithmic time windows, which does not follow from the support condition \eqref{Eq:Rrho}.
\end{remark}

\appendix

\section{Wasserstein and gradient estimates}\label{Sub:Appendix}
In this appendix, we collect elementary Wasserstein and gradient estimates for a family of operators induced by the marginal laws of probabilistically weak solutions to \eqref{Eq:OriginalMCV}.

For $\smash{x\in\R^d}$ and $t\in\R_+$, let $\smash{\sfh_t\delta_x\in\Prob(\R^d)}$ denote the law of such a solution started at $x$ at time $t$; we call the induced family $\sfh_\bullet$ \emph{Barenblatt flow}. Define the \emph{dual Baren\-blatt flow} acting on locally bounded Borel measurable functions $\smash{f\colon\R^d\to\R}$ by
\begin{align}
\sfp_tf(x) := \int_{\R^d} f\d\sfh_t\delta_x.
\end{align}
Note $\sfp_\bullet f$ is \emph{not} a solution of the Leibenson equation \eqref{Eq:Leibenson} with initial datum $f$, as \eqref{Eq:Leibenson} is nonlinear while $\sfp_\bullet f$ depends linearly on $f$. On the other hand, similar functionals play a prominent role in the study of McKean--Vlasov SDEs with possibly singular coefficients, cf.~e.g.~Ren--Wang \cite{ren-wang2019}, Huang--Wang \cite{huang-wang2022}, and the references therein.

An application of elementary coupling techniques to the Leibenson process is isometry of the Barenblatt flow with respect to Wasserstein distances.  For $q\in[1,\infty]$, we define the extended $q$-Wasserstein distance of $\smash{\mu,\nu\in\Prob(\R^d)}$ by
\begin{align}\label{Eq:Wasserstein}
    W_q(\mu,\nu):= \inf\{\Vert X-Y\Vert_{L^q(\Omega,\scrF,\PPP)} : \Law\,X  = \mu,\, \Law\,Y=\nu\};
\end{align}
more precisely, the infimum is taken over all complete probability spaces $(\Omega,\scrF,\PPP)$ and all random vectors $X$ and $Y$ on $\Omega$ with laws $\mu$ and $\nu$ under $\PPP$, respectively. Recall $W_q(\mu,\nu)$ depends nondecreasingly on the exponent $q$ and
\begin{align}\label{Eq:Winfty}
\begin{split}
    W_\infty(\mu,\nu) = \lim_{q\to\infty}W_q(\mu,\nu),
    \end{split}
\end{align}
cf.~e.g.~Villani \cite{villani2009}*{Rem. 6.6} and Kuwada \cite{kuwada2010}*{Lem.~3.2}. Moreover, recall $W_q$ is indeed an extended distance, which means it obeys the axioms of a metric except that it may attain the value $\infty$ \cite{villani2009}*{p. 106}. If $\smash{W_q(\mu,\nu)}$ is finite, there exists an \emph{optimal coupling}, i.e.~a pair $(X,Y)$ of random variables on a suitable complete probability space with $\Law\,X=\mu$ and $\Law\,Y=\nu$ such that $\smash{W_q(\mu,\nu)=\Vert X-Y\Vert_{L^q(\Omega,\scrF,\PPP)}}$ \cite{villani2009}*{Thm.~4.1}.
The following well-known property is specific to the linear structure of $\smash{\R^d}$.

\begin{lemma}[Wasserstein distance of translations]\label{Le:TransWasser} Given $\smash{v\in \R^d}$, let $T_v$ be the induced translation from \eqref{Eq:Translation}. Then for every $q\in[1,\infty]$ and every $\smash{\mu\in\Prob(\R^d)}$,
    \begin{align}
        W_q(\mu,(T_v)_\push\mu) = \vert v\vert.
    \end{align}
\end{lemma}

\begin{proof} By \eqref{Eq:Winfty}, it suffices to show the statement when $q$ is finite.

We first show ``$\leq$''. Let $X$ be a random vector with law $\mu$. The random vector $X+v$ has law $(T_v)_\push\mu$; in particular, $(X,X+v)$ is a coupling of $\mu$  and $(T_v)_\push\mu$. Thus,
\begin{align*}
    W_q(\mu,(T_v)_\push\mu) \leq \EEE\big[\vert X-(X+v)\vert^q\big]^{1/q}=\vert v\vert.
\end{align*}

Now we show ``$\geq$''. By nonnegativity of the $q$-Wasserstein distance, there is nothing to prove if $v$ is zero; thus, we may and will assume $v$ is not zero. The claim is also trivial if $\smash{W_q(\mu,(T_v)_\push\mu)=\infty}$; thus, we may and will assume $\smash{W_q(\mu,(T_v)_\push\mu)}$ is finite. In turn, this allows us to fix a $\smash{W_q}$-optimal coupling $(X,Y)$ of $\mu$ and $(T_v)_\push\mu$. Given $n\in\N$, let us define the function $\smash{\varphi_n\colon\R^d\to\R}$ by $\varphi_n(x) := (-n)\vee (n\wedge (x\cdot v/\vert v\vert))$. By Hölder's inequality and since the function $\varphi_n$ is $1$-Lipschitz continuous and bounded,
\begin{align*}
    W_q(\mu,(T_v)_\push\mu) &= \EEE\big[|X-Y|^q\big]^{1/q}\\
    &\ge \EEE\big[|X-Y|\big]\\
    &\ge \EEE\big[\varphi_n(Y) -\varphi_n(X)\big]\\
    &=\EEE\big[\varphi_n(X+v) -\varphi_n(X)\big],
\end{align*}
where the last identity holds since $Y$ and $X+v$ possess the same law. Using $1$-Lipschitz continuity again implies
\begin{align*}
    \sup\{\big\vert\varphi_n(X+v) -\varphi_n(X)\big\vert : n\in\N\} \leq \vert v\vert \quad\PPP\textnormal{-a.s.};
\end{align*}
consequently, Lebesgue's dominated convergence theorem yields
\begin{align*}
    W_q(\mu,(T_v)_\push\mu) \geq \lim_{n\to\infty} \EEE\big[\varphi_n(X+v) -\varphi_n(X)\big] = \vert v\vert.
\end{align*}
This is the desired lower bound.
\end{proof}
\cref{Le:TransWasser} combines with \cref{Le:Translation}, more precisely \eqref{Eq:Law}, to yield the following.

\begin{proposition}[Wasserstein isometry for Barenblatt flow]\label{Pr:WasserIso} For every $q\in[1,\infty]$, every $\smash{x,y\in\R^d}$, and every $t\in\R_+$, the Barenblatt flow $\sfh$ obeys
\begin{align}
W_q(\sfh_t\delta_x,\sfh_t\delta_y) = \vert x-y\vert.
\end{align}
\end{proposition}

Noting  $W_q(\delta_x,\delta_y) = \vert x-y\vert$ for every $q\in[1,\infty]$ and every $\smash{x,y\in\R^d}$, the preceding result states $\sfh_t$ is a $W_q$-isometry when acting on Dirac masses for every $t\in\R_+$.

\begin{remark}[Euclidean heat kernel] With the same argument based on the explicit form of the classical heat kernel, the preceding equality holds for the dual heat flow on $\smash{\Prob(\R^d)}$. This fact is well-known, but we could not locate an explicit reference.
\end{remark}

We now show gradient estimates for the dual Barenblatt flow $\sfp$. As the latter is defined by convolution against a translation-invariant and spatially compactly supported probability kernel, these would follow in principle from the identity $\nabla \sfp_t f=\sfp_t\nabla f$ on $\smash{\R^d}$ for  every $C^1$-function $f\colon \R^d\to\R$ and every $t\in\R_+$ plus  Jensen's inequality. (In particular, the gradient estimate below is sharp.) Anticipating later generalizations, we will instead use Kuwada's duality argument \cite{kuwada2010}*{Prop. 3.1}, whose proof we partly record for convenience.

\begin{theorem}[Gradient estimate for dual Barenblatt flow]\label{Th:DualBarFlo} For every given $q'\in [1,\infty)$, every $C^1$-function $f\colon\R^d\to\R$, and every $t\in\R_+$, the dual Barenblatt flow $\sfp$ obeys
\begin{align}
\vert \nabla \sfp_tf(x)\vert^{q'} \leq \sfp_t\big[\vert\nabla f\vert^{q'}\big](x).
\end{align}
\end{theorem}

\begin{proof} By Jensen's inequality, it suffices to consider the case when $q'$ is $1$. Given distinct points $\smash{x,y\in\R^d}$, by \cref{Pr:WasserIso} there exists a $W_\infty$-optimal coupling $(X,Y)$ of $\sfh_t\delta_x$ and $\sfh_t\delta_y$, which in particular obeys $\vert X-Y\vert\leq \vert x-y\vert$ $\PPP$-a.s.~by \eqref{Eq:Winfty}. Given $R>0$, we define the locally bounded function $\smash{G_R\colon \R^d\to\R}$ by
\begin{align}
    G_R(x) := \sup\!\Big\lbrace\frac{\vert f(z)-f(x)\vert}{\vert z-x\vert} : z\in B_R(x)\setminus\{x\}\Big\rbrace
\end{align}
which is universally measurable, cf.~e.g.~Ambrosio--Gigli--Savaré \cite{ambrosio-gigli-savare2014-riemannian}*{Lem.~2.5}. Thus,
\begin{align}
    \vert\sfp_tf(x)-\sfp_tf(y)\vert &\leq \EEE\big[\vert f(X)-f(Y)\vert\big]\\
    &= \EEE\Big[\frac{\vert f(X)-f(Y)\vert}{\vert X-Y\vert}\,\vert X-Y\vert\,\One_{\{X\neq Y\}}\Big]\\
    &\leq \vert x-y\vert\,\EEE\big[G_{\vert x-y\vert}(X)\big].
\end{align}
The latter expectation depends only on the law of  $X$. Since the latter has compact support and $f$ is Lipschitz continuous on it, Lebesgue's dominated convergence theorem implies
\begin{align}
    \lim_{y\to x}\EEE\big[G_{\vert x-y\vert}(X)\big] = \EEE\big[\vert\nabla f\vert(X)\big] = \sfp_t\vert\nabla f\vert(x).
\end{align}
On the other hand, since
\begin{align}
    \lim_{y\to x} \frac{\vert\sfp_tf(x)-\sfp_tf(y)\vert}{\vert x-y\vert} = \vert\nabla\sfp_tf\vert(x),
\end{align}
the proof is terminated.
\end{proof}

\addtocontents{toc}{\protect\setcounter{tocdepth}{1}}

\subsection*{Declaration of GenAI and AI-assisted technologies in the writing process} Following the recommendations of the \emph{Leiden Declaration on Artificial Intelligence and Mathematics} (\url{https://leidendeclaration.ai}), the author discloses that large language models (OpenAI's ChatGPT in the initial stage, Anthropic's Claude afterwards) were used in the preparation of this article. The conceptual framework, the main ideas, and the form of the central results are due to the author and were provided to the models as input. The models were employed to carry out routine computations and assisted with the typesetting and wording of the manuscript. Throughout this process, the author repeatedly intervened, corrected the output where necessary, and verified all arguments in detail. The responsibility for the correctness and adequacy of all arguments and results, as well as for the completeness and accuracy of citations, remains exclusively with the author.

\bibliographystyle{amsrefs}
\bibliography{library}
\end{document}